\documentclass[10pt]{amsart}
\usepackage{amsmath}
\usepackage{amssymb}
\usepackage{amsthm}
\usepackage{mathrsfs}
\usepackage{comment}
\usepackage{hyperref}

\usepackage[all,cmtip]{xy}\usepackage{xcolor}
\usepackage{enumerate}
\usepackage{bm}
\usepackage{dsfont}
\usepackage{mathtools}
\usepackage[cal=euler]{mathalfa}
\usepackage[top=1in, bottom=1.25in, left=1.25in, right=1.25in]{geometry}
\usepackage{parskip}

\hypersetup{colorlinks=true,linkcolor=magenta,citecolor=blue}

\DeclareMathAlphabet{\mathpzc}{OT1}{pzc}{m}{it}

\theoremstyle{plain}
\newtheorem{theorem}{Theorem}[section]
\newtheorem*{theorem*}{Theorem}

\newtheorem{lemma}[theorem]{Lemma}
\newtheorem*{claim*}{Claim}
\newtheorem{proposition}[theorem]{Proposition}

\newtheorem{corollary}[theorem]{Corollary}

\theoremstyle{definition}

\newtheorem{remark}[theorem]{Remark}

\newcommand{\ignore}[1]{}
\newcommand{\myframe}[1]{\begin{tabular}{|p{13.5cm}|}\hline #1\\\hline\end{tabular}}

\begin{document}
\setlength{\parindent}{0pt}

\ignore{
\documentclass[10pt]{amsart}
\usepackage[english]{babel}
\usepackage{amssymb}
\usepackage{amsmath}
\usepackage{srcltx}
\usepackage{lscape}

\thispagestyle{empty}

\textheight 21.5cm
\textwidth 14cm
\topmargin -0.6cm
\oddsidemargin 1cm
\evensidemargin 1cm

\font\fett=cmbx12 scaled \magstep2

\renewcommand{\baselinestretch}{1.3}
\newcommand{\myframe}[1]{\begin{tabular}{|p{13.5cm}|}\hline #1\\\hline\end{tabular}}

\renewcommand{\baselinestretch}{1.2}
\newcommand{\can}{\overline{\phantom{x}}}

\newtheorem{dummy}{Dummy}

\newtheorem{lemma}[dummy]{Lemma}
\newtheorem{theorem}[dummy]{Theorem}
\newtheorem{proposition}[dummy]{Proposition}
\newtheorem{corollary}[dummy]{Corollary}
\newtheorem{notation}[dummy]{Notation}
\newtheorem*{fact}{Fact}

\theoremstyle{definition}
\newtheorem{definition}{Definition}
\newtheorem{conjecture}{Conjecture}
\newtheorem{example}[dummy]{Example}
\newtheorem{question}{Question}
\newtheorem{problem}{Problem}
\newtheorem{remark}[dummy]{Remark}

\newcommand{\tl}{\widetilde{\tau}}
}

\date{\today}

\author{S. Pumpl\"un}
\email{susanne.pumpluen@nottingham.ac.uk}
\address{School of Mathematical Sciences\\
University of Nottingham\\
University Park\\
Nottingham NG7 2RD\\
United Kingdom
}




\title[Division algebras  isotopic to a cyclic Galois field extension]
{A classification of the division algebras that are isotopic to a cyclic Galois field extension}

\begin{abstract}

We classify all  division algebras that are principal Albert isotopes of a cyclic Galois field extension of degree $n>2$ up to isomorphisms.
 We achieve a ``tight'' classification when the cyclic Galois field extension is cubic. The classification is ``tight'' in the sense that the list of algebras has features that make it easy to distinguish non-isomorphic ones.
\end{abstract}

\maketitle

%
\section*{Introduction}
%

Classifications of classes of nonassociative division algebras of dimension $n$ over some base field $F$  heavily depend on the choice of $F$ and on $n$, take $n=3$ as example: if $F$  is an algebraically closed or a real closed field, then
 there are no three-dimensional nonassociative division algebras over $F$.
  Over finite fields, however, three-dimensional nonassociative division algebras  exist.

  The first examples of three-dimensional division algebras are due to Dickson who constructed three-dimensional unital commutative nonassociative algebras already in 1905 \cite{D1, D2, D3, D4}. Kaplansky then developed their structure theory in a more modern language \cite{KI, KII}. Over finite fields of characteristic not two, every three-dimensional commutative division algebra is a field or one of these ``Dickson'' algebras; all are  isotopic to generalized twisted fields
 \cite{M1, M2}.

In this paper, we will investigate the $n$-dimensional algebras over a field $F$ that are principal Albert isotopes of cyclic field extensions, 
with a special emphasis on $n=3$.

  Let $V$ be an $n$-dimensional $F$-vector space and $ {\rm Alg}(V)$ be the set of non-unital algebra structures on $V$.
  If $A \in {\rm Alg}(V)$ is \emph{regular}  (i.e.,  the left and right multiplication by some element in $A$ are invertible) then the orbit of $A$ under
$G_0(V)={\rm Gl}(V)\times {\rm Gl}(V)$ contains a unital algebra.  If this algebra is associative or an octonion algebra, then it  is unique up to isomorphism  (Theorem \ref{thm:1.10}). We call it the \emph{unital heart} of $A$. As a first step we obtain
   a rough ``classification'' of  all those regular algebras which have different associative unital hearts, and which have different octonion algebras as their unital heart  (Theorem \ref{thm:1.10}).  We  proceed to give necessary and sufficient criteria for two algebras to be isomorphic, which have the same commutative associative  unital heart (Theorem \ref{thm:1.12}).

 In this paper we focus on algebras whose unital heart is a cyclic Galois field extension $K/F$  of dimension $n$. Nonassociative $n$-dimensional division algebras isotopic to cyclic field extensions are
 understood up to isomorphism over those base fields, where we have a full understanding of the cyclic field extensions of degree $n$: Given a fixed cyclic field extension $K$ of degree $n$ over $F$, we can list all the types of  division algebras over the field $F$ whose unital heart is $K$, but only achieve a full classification for some subtypes (Theorem \ref{thm:cyclicGalois}). The situation improves slightly when $n$ is prime (Theorem \ref{thm:cyclicGaloisprime}).

The Classification Theorem for algebras with a cubic field extension as unital heart is ``tight'' in the sense of Petersson for several cases (Theorem \ref{thm:classcubic}). We then give necessary and sufficient criteria
for any two  algebras of dimension three with the same unital heart to be isomorphic.

The content of the paper is  as follows:
After collecting the basic terminology (Section \ref{sec:prel}), we describe the bijective linear maps for Galois $C_n$-algebras in Section \ref{sec:invertiblemaps}. We then consider regular algebras which have  a Galois $C_n$-algebra as their unital heart (Section \ref{sec3}), and
achieve a classification of the three-dimensional division algebras which have  a cubic Galois field extension as unital heart in Section \ref{sec:4}.
With the same methology, we then investigate the algebras whose unital heart is a cyclic Galois field extension of degree $n$ in Section \ref{sec5}. In this generality, we are still able to list non-isomorphic classes of algebras, but are not able to parametrize the algebras to the extent we were able to in several subcases for the $n=3$ case.

There are a range of papers that attempt to classify certain subfamilies of nonassociative  algebras of small dimension up to isomorphism, often over $\mathbb{C}$, $\mathbb{R}$, or finite fields, and often employing structural constants: unital associative algebras over algebraically closed fields of dimension less than five are classified  in \cite{G};
five-dimensional unital associative algebras over an algebraically closed field of characteristic not two are classified up  in \cite{Ma}, both approaches use structural constants. A classification of two-dimensional nonassociative algebras over any base field employing structure constants can be found in \cite{B}. While there are good arguments for using structure constants, eg. see \cite{Ka, KV},
we have avoided them in our classification, as we believe the structure of the algebras becomes more visible this way.

 Petersson \cite{P}
 achieved a full classification of all nonassociative two-dimensional algebras  up to isomorphism that was applied to both finite and real base fields \cite{HP, PS}.
In this paper, we  successfully generalized Petersson's methodology to certain classes of $n$-dimensional algebras.

   Regular
 algebras whose unital heart is a $C_n$-Galois algebra can be listed in a similar way, if desired.
 Singular (i.e., not regular) algebras can be partly investigated by a direct approach following \cite{P}  as well.
  Partial results can be obtained but are not treated here.

%
%

\section{Singular and regular algebras and the unital heart} \label{sec:prel}

\subsection{Isotopes and Kaplansky's trick}

 Let $R$ be a unital commutative ring $R$, and $M$ an $R$-module. Write ${\rm Alg}(M) = {\rm Hom}_R(M\otimes_R M,M)$ for the $R$-module of non-associative $R$-algebra structures on $M$. Given $A\in {\rm Alg}(M)$, we write $xAy$ for the product of $x,y\in M$ in the algebra $A$ instead of simply juxtaposition, when it is not ad hoc clear which algebra multiplication is considered.
If $A$ is a unital associative algebra, then we denote the group of invertible elements in $A$ by $A^\times$.

 Define
  $S_0(M) = {\rm End}_R(M) \times {\rm End}_R(M)$. If we view ${\rm End}_R(M)$ as a multiplicative monoid in the natural way, then $S_0(M)$ is a multiplicative monoid that contains $G_0(M) = {\rm Gl}(M) \times {\rm Gl}(M)$ as a subgroup.

For $f,g,h\in {\rm End}_R(M)$ we define the algebra $A^{(f,g,h)}$, called a
{\it homotope} of $A$, as $M$ together with the new multiplication
$$xA^{(f,g,h)}y=h(f(x)A g(y))$$
for $x,y\in M$. Two algebras $A, A'\in {\rm Alg}(M)$ are called {\it homotopic}, if $A' = A^{(f,g,h)}$ for some $f,g,h \in {\rm End}(M)$, and {\it isotopic}, if $A' = A^{(f,g,h)}$ for some $f,g,h \in {\rm Gl}(M)$ (
 in particular, if  $f=g=h^{-1}$ then $A\cong A'$ are isomorphic).
We call  $A^{(f,g)} =A^{(f,g,id)}$ a \emph{principal Albert homotope} of $A$, and if $(f,g)\in G_0(M)$ a \emph{principal Albert isotope} of $A$.

The monoid $S_0(M)$ acts on ${\rm Alg}(M)$ from the right by principal Albert homotopes via
\begin{align*}
xA^{(f,g)}y = f(x)Ag(y) &&(x,y \in M,\;A \in {\rm Alg}(M),\; (f,g) \in S_0(M))
\end{align*}

 We know that $(A^{(f,g)})^{op}=(A^{op})^{(f,g)}$.

 An algebra $A \in {\rm Alg}(M)$  is called \emph{left regular}, if there exists $u\in A$ such that $L_A(u)\in {\rm Gl}(M)$,
  otherwise it is \emph{left singular}, and
   \emph{right regular}  if there exists $u\in A$ such that $R_A(u)\in {\rm Gl}(M)$, otherwise it is \emph{right singular}.
  If $A$  is both left and right regular, it is called \emph{regular}, otherwise \emph{singular}. Clearly, $A$ is
left regular if and only if $A^{op}$ is right regular. It also follows easily that, if $A\in {\rm Alg}(M)$
is regular, so are the isotopes $A^{(f,g)}$ for all $(f,g) \in G_0(V)$.

Let $F$ be a field.  From now on, let $V$ denote an $n$-dimensional $F$-vector space.
Let $0\not=A\in {\rm Alg}(V)$. Then $A$ is called a \emph{division algebra}  if for any
$a\in A$, $a\not=0$, the left and right multiplication with $a$, $L_a: A\rightarrow A$, $L_a(x)=ax$,
and  $R_a: A\rightarrow A$, $R_a(x)=xa$, are bijective.
If it is not clear which algebra structure on $V$ is considered, we  denote the left and right multiplication with $a\in V$ by $L_A(a)$ and $R_A(a)$.
If $A$ has finite dimension over $F$, then $A$ is a division algebra if and only if $A$ has no zero divisors \cite[pp. 15, 16]{Sch}.
 If $A$ is a division algebra with underlying $F$-vector space $V$ and $f,g,h\in {\rm Gl}(V)$, then $A^{(f,g,h)}$ is a division algebra.
Note that every division algebra over a field $F$ is a regular algebra, and that the principal Albert isotope of any division algebra again is a division algebra.

The following extends \cite[Lemma 1.3]{P}, its proof remains the same:

\begin{lemma}
For $A,B\in {\rm Alg}(M)$, $(f,g)\in S_0(M)$, every isomorphism $\varphi:A\rightarrow B$ is also an isomorphism $\varphi: A^{(f,g)}\rightarrow B^{(f',g')}$ where $f'=\varphi f\varphi^{-1}$, and $g'=\varphi g\varphi^{-1}$.
\end{lemma}

Division algebras are principal Albert isotopes of unital division algebras, 
more generally so are regular algebras \cite[1.5]{P}:

\begin{proposition} (Kaplansky's Trick)\label{prop:1.5} 
Let $V$ be an $n$-dimensional $F$-vector space, $A\in {\rm Alg}(V)$ be a regular algebra, and choose $u, v \in V$ such that $f=R_A(v)$ and $g=L_A(u)$ are bijective.
Then $B=A^{(f^{-1},g^{-1})}$ is a unital algebra over $F$ with unit element $e=uAv$ and $A=B^{(f,g)}$.
\end{proposition}

\begin{lemma} \label{le:1.9}
Let $K\in {\rm Alg}(M)$ be a unital commutative associative algebra and $(f,g)\in G_0(M)$.  Write $L=L_K$ and $R=R_K$. Then the following are equivalent:
\\ (i) $K^{(f,g)}$ is a unital algebra over $R$.
\\ (ii) $K^{(f,g)}$ is a \emph{Jordan isotope} of $K$, i.e. there exists $w\in K^\times$ such that $K^{(f,g)}=K^{(w)}$, with $K^{(w)}$
defined via $xK^{(w)}y=xwy$ for all $x,y\in V$.
\\ (iii) $K^{(f,g)}\cong K$.
\\ Moreover, in this case $w$ satisfies $w^{-1}=1_{K^{(f,g)}}$ and is uniquely determined, and there are $u,v\in K$ such that
$w=uv$, $f=L(u)$, $g=R(v)$.
\\ A map $\varphi: K^{(f,g)}\rightarrow K$ is an isomorphism if and only if $\varphi=\sigma L(w)$ for some $\sigma\in {\rm Aut} (K)$.
\end{lemma}

This is the generalization of \cite[Lemma 1.9]{P} which was stated only for $n=2$. Its proof holds verbatim for any $n$-dimensional unital commutative associative $R$-algebra $K$.

Let $A \in {\rm Alg}(V)$ be a unital alternative $F$-algebra and suppose $p,q \in A$ are invertible. Following McCrimmon \cite{McC}, the product
\[
x\,._{p,q}\,y := (xp)(qy)
\]
for $x,y \in V$ makes $A$ into a new algebra denoted $A^{(p,q)} \in {\rm Alg}(V)$ which is again unital alternative; its unit element is given by $1^{(p,q)} = (pq)^{-1}$. Two unital alternative algebras $A,B \in {\rm Alg}(V)$ are called \emph{$(p,q)$-isotopic}, if $A^{(p,q)} \cong B$ for some invertible elements $p,q \in A$. $(p,q)$-isotopy is an equivalence relation that breaks down to isomorphism if $A$ is associative or an octonion algebra.

Let $A \in {\rm Alg}(V)$ be a unital alternative algebra, $f,g \in {\rm Gl}(V)$ and $B = A^{(f,g)}$. Then the following conditions are easily seen to be equivalent:
\begin{itemize}
\item [(i)] $B$ is unital.

\item [(ii)] There are invertible elements $p,q \in A$ such that $B \cong A^{(p,q)}$.
\end{itemize}
 Let $B \in {\rm Alg}(V)$ be regular. Then $A \in {\rm Alg}(V)$ is called a \emph{unital heart} of $B$ if $A$ is unital and there exist $f,g \in {\rm Gl}(V)$ such that  $A^{(f,g)} \cong B$. By Kaplansky's trick, unital hearts always exist, and by the above equivalence, if one of them is alternative (resp. associative, commutative associative, or an octonion algebra), all of them are, in which case the unital heart is unique up to $(p,q)$-isotopy (resp. isomorphism):

\begin{theorem} (The Unital Heart)\label{thm:1.10} 
For all $A\in {\rm Alg}(V)$, the following statements are equivalent:
\\ (i) $A$ is regular.
\\ (ii) There exist a unital algebra $K\in {\rm Alg}(V)$ and $f,g\in {\rm Gl}(V)$ such that $A\cong K^{(f,g)}$.
\\
 If $K$ is alternative, then $K$ is unique up to $(p,q)$-isotopy. In particular, if $K$ is associative or an octonion algebra, then  $K$ is unique up to  isomorphism  and is called \emph{the} unital heart of $A$.
 \end{theorem}

This generalizes \cite[Proposition 1.10]{P} which was formulated for $n=2$.
It remains an open problem how or if this approach might work for regular algebras that are Albert isotopes of a unital algebra that is not an alternative algebra.

\begin{proof}
The proof that (i) and (ii) are equivalent given in \cite[Proposition 1.10]{P} holds verbatim for any $n$-dimensional unital algebra $K$. It remains to show that $K$ is unique up to $(p,q)$-isotopy when $K$ is alternative: Let $K,K'\in {\rm Alg}(V)$  be unital and  alternative with
$A\cong K^{(f,g)}$ and $A\cong K^{(f',g')}$ for some  $f,g,f',g'\in {\rm Gl}(V)$.
Then there are $f'',g''\in {\rm Gl}(V)$  such that $K'\cong K^{(f'',g'')}$, therefore
$K\cong K'^{(p,q)}$ for some invertible $p,q$ in $K'$. 
 The two  alternative unital algebras $K,K'$ thus are $(p,q)$-isotopic.
 By \cite{McC}, this means $K\cong K'$  if one of them is associative, or an octonion algebra.
\end{proof}

\begin{theorem} (Isomorphism Criterium)\label{thm:1.12} 
Let $K\in {\rm Alg}(M)$ be a unital commutative associative algebra and
  $(f,g) \in G_0(V)=G_0(M)$, $(f^\prime,g^\prime) \in S_0(M)$.
  Write $L=L_K$ and $R=R_K$. Then the following statements are equivalent:
\\ (i) $\varphi:K^{(f,g)}\rightarrow K^{(f',g')}$ is an isomorphism.
\\ (ii) There exist $u,v\in K^\times$, $\sigma\in {\rm Aut}(K)$ such that
$$f'=L(v^{-1})\sigma f \sigma^{-1}L(uv), \quad g'=L(v^{-1})\sigma g \sigma^{-1}L(uv),\quad \varphi=L((uv)^{-1})\sigma.$$
\end{theorem}

This generalizes \cite[1.12]{P} to any dimension, the proof is verbatim the same. Note that
the proof  of $(ii)\Rightarrow (i)$ in \cite[1.12]{P} needs that $K$ is associative and that $u\in {\rm Comm}(K)$.

\begin{theorem} (Isomorphism Criterium for Octonion Algebras) \cite[Proposition 4]{D}\label{thm:1.12quat}
Let $K\in {\rm Alg}(V)$ be a quaternion algebra or an octonion algebra over a field $F$ and $(f,g)$, $(f',g')\in G_0(V)$. Write $L=L_K$ and $R=R_K$. Then the following statements are equivalent:
\\ (i) $\varphi:K^{(f,g)}\rightarrow K^{(f',g')}$ is an isomorphism.
\\ (ii) $\varphi$ is a similitude of the norm $N_K$ of $K$ and  there are similitudes $\varphi_i$ of $N_K$ with $\varphi(xy)=
\varphi_1(x)\varphi_2(y)$ for all $x,y\in K$ and
$$f'=\varphi_1 f \varphi^{-1}, \quad g'=\varphi_2 g \varphi^{-1}$$
or equivalently,
$$f'=R_{\varphi_2(1_K)}^{-1}\varphi f \varphi^{-1}, \quad g'=L_{\varphi_1(1_K)}^{-1}\varphi g \varphi^{-1}.$$
\end{theorem}

As a side remark, we observe:

\begin{corollary} \label{cor:1.13} 
Let $K\in {\rm Alg}(V)$ be a unital associative 
algebra, such that its norm $N(x)=\det L(x)$  is a form of degree $n$ permitting composition, and $(f,g)\in G_0(V)$.
\\ (i) The pair
$$(\det (f)\,{\rm mod}\, N(K^\times), \det (g)\,{\rm mod}\, N(K^\times))$$
is an invariant of $K^{(f,g)}$ in $(F^\times/N(K^\times),F^\times/N(K^\times))$.
\\ (ii)  For all $a,b\in F^\times$, we have $K^{(f,g)}\cong K^{(af,bg)}$.
\end{corollary}

\begin{proof} (i) The proof generalizes the one of \cite[Corollary 1.13]{P}: let $f,g,f',g'\in {\rm Gl}(V)$. We compare the determinants in Theorem \ref{thm:1.12} (ii) and obtain
$$\det (f')=\det(L(v^{-1})\det(\sigma )\det(f )\det(\sigma^{-1})\det(L(uv))=\det(f )\det(L(v^{-1})\det(L(uv))$$
and analogously,
$\det (g')=\det(g )\det(L(v^{-1})\det(L(uv)).$
Using that  $N(x)=\det L(x)$ for $x\in V$  we get
$\det (f')=\det(f )N(v^{-1})N(uv)$ and $\det (g')=\det(g )N(v^{-1})N(uv).$
 Since $N$ permits composition
  it follows that $\det (f')=\det(f )N(u)$ and $\det (g')=\det(g )N(u),$ and we obtain the assertion.
  \\ (ii) This generalizes \cite[Corollary 1.14]{P}.  The proof is identical to the one of \cite[Corollary 1.14]{P} and uses Theorem \ref{thm:1.12}.
\end{proof}

\begin{corollary} \label{cor:1.14} 
 Let $K\in {\rm Alg}(V)$ be a quaternion or octonion algebra  over a field $F$ and $(f,g)\in G_0(V)$. Then
$$K^{(f,g)}\cong K^{(af,bg)}$$
for all $a,b\in F^\times$.
\end{corollary}

\begin{proof}
 Let $a,b\in F^\times$, $c=ab$, then $\varphi_1(x)=ax$, $\varphi_2(x)=bx$, $\varphi_1(x)=cx$ are norm similarities of $N$ and $\varphi(xy)=\varphi_1(x)\varphi_2(y)$ for all $x,y\in K$. We have
$f'(x)=\varphi_1 (f (\varphi^{-1}(x))))=\varphi_1 (f (c^{-1}x))=ac^{-1}f(x)=b^{-1}x$ and $ g'(x)=\varphi_2( g (\varphi^{-1}(x))))=
\varphi_2( g (c^{-1}x))=bc^{-1}f(x)=a^{-1}x$.  By  Theorem \ref{thm:1.12quat}, then $K^{(f,g)}\cong K^{(b^{-1}f,a^{-1}g)}$.
Choosing $a^{-1}$ and $b^{-1}$ instead of $b,a$ in the calculation yields the assertion.
\end{proof}

We can now ``sort''  nonassociative regular algebras over $F$ of dimension $n$ by their unital heart. This breaks up the classification of regular algebras with alternative unital hearts into two steps: (i) the classification of
the possible alternative unital hearts of dimension $n$ up to $(p,q)$-isotopy,
 and then (ii), the classification of the algebras with the same alternative unital heart.
Unlike in the $n=2$ case, however, even the class of unital associative algebras is now in general formidable, and does not only contain commutative associative algebras anymore as it does for $n=2$.

Any classification  of the unital associative algebras $K\in {\rm Alg}(V)$  of dimension $n$, and of the octonion algebras $K\in {\rm Alg}(V)$,
yields a classification of the possible unital hearts of dimension $n$ up to isomorphism.

\begin{remark}
The restriction to unital commutative associative algebras, respectively, to octonion algebras as unital hearts is clearly limiting the scope of the classification we can achieve with our method.  Even in small dimensions, and over an algebraically closed field of characteristic not two, 19 of the 59 unital associative algebras of dimension 5 listed up to isomorphism  in \cite{Ma} are commutative. For a complete lists of low dimensional complex associative nonunital algebras see
\cite{RRB}.
\ignore{which was motivated by trying to classify diassociative algebras, an associative version of Leibniz algebras, which were introduced
in the search of an “obstruction” to the periodicity of algebraic $K$-theory, and have applications in  classical geometry, non-commutative geometry and physics.}
\end{remark}

%
%

\section{Invertible linear maps}\label{sec:invertiblemaps}

\subsection{Invertible linear maps for Galois $C_n$-algebras}

 A commutative associative algebra  $K\in {\rm Alg}(V)$ of dimension $n$ is called {\it \'{e}tale},
if $K\cong F_1\times\dots\times F_r$ for some finite separable field extensions $F_1,\dots, F_r$ of a field  $F$. Let $K$ be endowed with an action by
 a finite group $G$ of $F$-automorphisms. Then $K$ is called a \emph{$G$-algebra} \cite[(18.B.), p.~287]{KMRT}.
An \'{e}tale algebra $K$ over $F$  of dimension $n$ for which the order of $G$ equals $n$, then $K$ is a {\it Galois $G$-algebra} over $F$,  if 
 ${\rm Fix}(G)=\{x\in K\,|\, g(x)=x \text{ for all } g\in G \}=F$ \cite[(18.15), p.~288]{KMRT}.
A Galois $G$-algebra structure on a field $K$ exists if and only if the field extension $K/F$ is Galois with Galois
group isomorphic to $G$  \cite[(18.16), p.~288]{KMRT}.

Let $V$ be an $n$-dimensional $F$-vector space. Let $C_n$ denote the cyclic group of order $n$. Let $K\in {\rm Alg}(V)$ be a Galois $C_n$-algebra, i.e.  a cyclic Galois field extension of $F$ of degree $n$ with Galois group  ${\rm Gal}(K/F)=\langle\tau \rangle$ or an $n$-dimensional \'etale algebra over $F$ with an automorphism  $\tau\in {\rm Aut}_F(K)$ of order $n$, i.e. where $G=\langle \tau\rangle$, with norm $N$ and trace $T$.
Take the cyclic algebra  $A_0=(K/F, \tau, 1)=K[t;\sigma]/(t^n-1)$ of degree $n$ with reduced norm $N_{A_0/F}: A_0 \rightarrow F$,
$N_{A_0/F}(x) = {\rm det}(R_x).$
 Consider $A_0$ as a left $K$-module with basis $\{1,t, \ldots, t^{n-1}\}$.

\begin{theorem} (cf. \cite[Proposition 2.2]{P} for $n=2$)\label{thm:2.2} 
Let $V$ be an $n$-dimensional $F$-vector space. Let $K\in {\rm Alg}(V)$ be a 
Galois $C_n$-algebra with generating automorphism $\tau$.
Then the map $(K/F, \tau, 1) \to {\rm End}_F(V)$,
\[
\sum_{i=0}^{n-1}x_it^i \mapsto \sum_{i=0}^{n-1}L(x_i)\tau^i;
\]
is an isomorphism  of central simple algebras. In particular, we can write
 $${\rm End}_F(V)=L(K) \oplus L(K)\tau \oplus \dots \oplus L(K)\tau^{n-1}$$
as a direct sum of $F$-subvector spaces and
$$\det (\sum_{i=0}^{n-1}L(x_i)\tau^i)=N_{A_0/F}(x_0 + x_1t + \cdots + x_{n-1}t^{n-1}) .$$
\end{theorem}

\begin{proof}
The map
$\sum_{i=0}^{n-1}x_it^i \mapsto \sum_{i=0}^{n-1}L(x_i)\tau^i$
is a
homomorphism. This is an immediate consequence of the way multiplication in $(K/F, \tau, 1)$ is defined; $(x_it^i)(x_jt^j) = (x_i\tau^i(x_j))t^{i+j}$ for all $i,j\in \{0,\dots,n-1\}$. 
A homomorphism between two central simple algebras of the same dimension is an isomorphism  of central simple algebras.

The second statement is now
obvious. The third follows from the fact that norms of central simple algebras are preserved
by isomorphisms, but was also independently previously observed in \cite{S13}.
\end{proof}

\begin{proposition} 
Assume we are in the situation of Theorem \ref{thm:2.2}. Suppose that $n=rs$ for some positive integers $s\geq 2$, $s\not=n$. Let $E={\rm Fix}(\tau^s)$ and view $K$  as an $r$-dimensional $E$-vector space. Then the map
\[
\sum_{i=0}^{r-1}x_{si}t^{si} \mapsto \sum_{i=0}^{r-1}L(x_{si})\tau^{si};
\]
induced by restricting the map $(K/F, \tau, 1) \to {\rm End}_F(V)$  in Theorem \ref{thm:2.2} to $B_0=(K/E, \tau^s, 1)$,
is an isomorphism  between the central simple algebras $(K/E, \tau^s, 1)$ and ${\rm End}_E(K)$ over $E$, where both algebras are now viewed as $F$-algebras. In particular, we can write
 $${\rm End}_E(K)=L(K) \oplus L(K)\tau^s \oplus \dots \oplus L(K)\tau^{s(r-1)}$$
as a direct sum of $E$-subvector spaces and
$$\det (L(x_0) + L(x_s)\tau^s + \dots + L(x_{s(r-1)})\tau^{s(r-1)})=N_{B_0/E}(x_0 + x_1t^s + \cdots + x_{s(r-1)}t^{s(r-1)}) .$$
\end{proposition}

\begin{proof}
The fact that the map is an  isomorphism  between the central simple algebras $(K/E, \tau^s, 1)$ and ${\rm End}_E(K)$ of degree $r$ over $E$ is shown analogously as in Theorem \ref{thm:2.2}. The algebra $(K/E, \tau^s, 1)$ is an $F$-subalgebra of  $(K/F, \tau, 1)$ and so is ${\rm End}_E(K)$  of ${\rm End}_F(V)$. Restricting an isomorphisms   $(K/F, \tau, 1) \to {\rm End}_F(V)$  to the subalgebra $(K/E, \tau^s, 1)$ yields exactly the isomorphism  $(K/E, \tau^s, 1) \to {\rm End}_E(K)$. 
\end{proof}

\begin{corollary}\label{cor:2.9}
Assume we are in the situation of Theorem \ref{thm:2.2}.
 Then every $f\in {\rm End}_F(V)$ can be uniquely written in the form
 \begin{equation}\label{iso}
 f= \sum_{i=0}^{n-1}L(y_i)\tau^i
 \end{equation}
 for some $(y_0,y_1,\dots,y_{n-1})\in K^n$ and lies in ${\rm Gl}(V)$  if and only if $N_{A_0}(y_0+y_1t+\dots+y_{n-1}t^{n-1})\not=0$.
 \\ If $n=rs$ for $s\geq 2$, $s\not=n$, $E={\rm Fix}(\tau^s)$, then  every $f\in {\rm End}_E(V)$ can be written in the form
 $$f=L(y_0)+L(y_s)\tau^s +  \dots + L(y_{s(r-1)})\tau^{s(r-1)}$$
 for some $(y_0,y_s,\dots,y_{s(r-1)})\in K^r$,
is an isomorphism  between $(K/E, \tau^s, 1)$ and ${\rm End}_E(K)$.
 \end{corollary}

\subsection{Invertible linear maps for any abelian Galois field extension}

We can push this approach a bit further: Let $V$ be an $n$-dimensional $F$-vector space.
Let $K\in {\rm Alg}(V)$ be a Galois field extension of $F$ of degree $n$ with abelian Galois group  $G={\rm Gal}(K/F)=\{\tau_1,\dots,\tau_n\}$.

Take the crossed product algebra  $A_0=(K/F, G, \Phi)$ of degree $n$ with the 2-cocycle  $\Phi$ defined by $\Phi(\sigma,\tau)=1$ for all $\sigma,\tau \in G$ and  reduced norm $N_{A_0/F}: A_0 \rightarrow F$, $N_{A_0/F}(x) = {\rm det}(R(x)).$
 Consider $A_0$ as a left $K$-vector space with basis $\{1,u_{\tau_1}, \ldots, u_{\tau_{n-1}}\}$.
  A straightforward calculation shows that the matrix  $R(x)$  of right multiplication by
 $x = x_0 + x_1 u_{\tau_1} + \cdots + x_{n-1}u_{\tau_{n-1}},$ $ x_i \in K$ in $A_0$ is given by
\[ R(x)=\left (\begin {array}{cccccc}
x_0 &  \tau_1(x_{n-1}) & ... &  \tau_{n-1}(x_{1})\\
x_1 & \tau_1(x_0) & ... &   \tau_{n-1}(x_{2})\\
x_2 & \tau_1(x_{1}) & ... 
 & \tau_{n-1}(x_{3})\\
...& ...  & ... & ...\\
x_{n-2} &  \tau_1(x_{n-3})  & ...&  \tau_{n-1}(x_{n-1})\\
x_{n-1} &  \tau_1(x_{n-2})  & ...& \tau_{n-1}(x_{0})\\
\end {array}\right ).\]

\begin{theorem} \label{thm:2.2Galois} 
 Let $V$ be an $n$-dimensional $F$-vector space.
Let $K\in {\rm Alg}(V)$ be a Galois field extension of $F$ of degree $n$ with abelian Galois group $G=\{\tau_1,\dots,\tau_n\}$.
The map $(K/F, G, \Phi) \to {\rm End}_F(V),$
$$
(x_0,\dots,x_{n-1})\mapsto \sum_{i=0}^{n-1}L(x_i)\tau^i,
$$
is an isomorphism  of central simple algebras. In particular, we can write
 $${\rm End}_F(V)=L(K) \oplus L(K)\tau_1 \oplus \dots \oplus L(K)\tau_{n-1}$$
as a direct sum of $F$-subvector spaces and
$$\det ( \sum_{i=0}^{n-1}L(x_i)\tau^i)=N_{A_0/F}(x_0 + x_1 u_{\tau_1} + \cdots + x_{n-1}u_{\tau_{n-1}}) .$$
\end{theorem}

\begin{proof}
The map
$(x_0,\dots,x_{n-1})\mapsto  \sum_{i=0}^{n-1}L(x_i)\tau^i$
is a homomorphism.
This is an immediate consequence of the structure of the matrix $R(x)$ which defines the right multiplication in $A_0$
and the fact that we have  $\tau_i R(x)=R(\tau_i (x))\tau_i$ for all $x\in K$ and suitable integers $i$.
A homomorphism between two central simple algebras of the same dimension is an isomorphism  of central simple algebras.

The second statement is now
obvious. The third follows again from the fact that norms of central simple algebras are preserved
by isomorphisms.
\end{proof}

\begin{corollary}\label{cor:2.9'}
In the situation of Theorem \ref{thm:2.2Galois}, every $f\in {\rm End}_F(V)$ can be uniquely written in the form
 \begin{equation}\label{iso}
 f= \sum_{i=0}^{n-1}L(y_i)\tau^i
 \end{equation}
 for some $(y_0,y_1,\dots,y_{n-1})\in K^n$ and lies in ${\rm Gl}(V)$  if and only if $N_{A_0}(y_0 + y_1 u_{\tau_1} + \cdots + y_{n-1}u_{\tau_{n-1}})\not=0$.
 \end{corollary}

  We often write $N_{A_0/F}(y)=N_{A_0/F}(y_0+y_1t+\dots+y_{n-1}t^{n-1})$ from now on.

 Analogous results holds for Galois $G$-algebras.

\section{Regular algebras with a Galois $C_n$-algebra as unital heart}\label{sec3}

 Let $V$ be an $n$-dimensional $F$-vector space.
Let $K\in {\rm Alg}(V)$ be a Galois $C_n$-algebra. Put $N=N_{K/F}$.
Define the set of idempotents ${\rm Id}(K)=\{d\in K\,|\, d^2=d\}$, and note that  ${\rm Id}(K)=\{0,1\}$ if $K$ is a field. 

\begin{lemma}\label{le:2.8} 
 Let $K \in {\rm Alg}(V)$ be a Galois $C_n$-algebra, with distinguished $F$-automorphism $\tau$ of order $n$. Write $f,g,f^\prime,g^\prime \in {\rm Gl}(V)$ as
\[
f = \sum_{i=0}^{n-1}L(y_i)\tau^i,\;g = \sum_{i=0}^{n-1}L(z_i)\tau^i,\;f^\prime = \sum_{i=0}^{n-1}L(y_i^\prime)\tau^i,\;g^\prime = \sum_{i=0}^{n-1}L(z_i^\prime)\tau^i
\]
with $y_i,z_i,y_i^\prime,z_i^\prime \in K$ for $0 \le i \le n - 1$.
 \\ (i) We have $$K^{(f,g)} \cong K^{(f^\prime,g^\prime)} $$
 if and only if there are $u,v \in K^\times$, $\sigma \in {\rm Aut}_F(K)$ such that
 $$y_i^\prime = \tau^i(u)\tau^i(v)v^{-1}\sigma(y_i)$$
for $0 \le i \le n - 1$ and
$$g'=L(u^{-1})\sigma g\sigma^{-1}L(uv).$$
 This is equivalent to saying that
there are $u,v \in K^\times$, $\sigma \in {\rm Aut}_F(K)$ such that
$$y_i^\prime = \tau^i(u)\tau^i(v)v^{-1}\sigma(y_i), \quad z_i^\prime = \tau^i(u)\tau^i(v)u^{-1}\sigma(z_i)$$
for $0 \le i \le n - 1$.
 \\ (ii) Let $K$ be a field and $y_0=d,y_0'=d'\in {\rm Id}(K)$. If the statements in (i) hold then  $d'=\sigma(d)$ and $d=1$ forces $d'=u=1$.
\\ (iii) If  the statements in (i) hold then for all $i\in \{0,\dots,n-1\}$,  $y_i$ is invertible  if and only if $y_i'$ is, and nonzero if and only so is $y_i'$.
 \end{lemma}

 \begin{proof}
 (i) By Theorem \ref{thm:1.12},
 $K^{(f,g)}\cong K^{(f',g')}$
  if and only if there are $u,v\in K^\times$,
 $\sigma\in {\rm Gal}(K/F)$, such that $g'=L(u^{-1})\sigma g\sigma^{-1}L(uv)$ and
 $$L(y_0')+L(y_1')\tau + L(y'_2)\tau^2 + \dots + L(y'_{n-1})\tau^{n-1}$$
 $$= L(v^{-1}) \sigma( L((y_0)+L(y_1)\tau + L(y_2)\tau^2 + \dots + L(y_{n-1})\tau^{n-1} )\sigma^{-1} L(uv)$$
 $$= L(v^{-1})    ( L(\sigma(y_0))+L(\sigma(y_1))\tau + L(\sigma(y_2))\tau^2 + \dots + L(\sigma(y_{n-1}))\tau^{n-1} ) L(uv) $$
 $$= L(u\sigma(y_0))+L(\tau(u)\tau(v) v^{-1}\sigma(y_1))\tau +
 L(\tau^2(u)\tau^2(v) v^{-1}\sigma(y_2))\tau^2 + $$
 $$\dots + L(\tau^{n-1}(u)\tau^{n-1}(v) v^{-1}\sigma(y_{n-1}))\tau^{n-1} . $$
 Compare components as in (\ref{iso})  to get
 $$y_i^\prime = \tau^i(u)\tau^i(v)v^{-1}\sigma(y_i), $$
for $0 \le i \le n - 1$. The second equivalence is shown analogously.
 \\ (ii) Let $y_0=d$, $y_0'=d'\in {\rm Id}(K)$. If the statements are fulfilled, we have $u \sigma(d)=d'=d'^2=u^2\sigma(d)^2=u^2 \sigma(d)$. Since $K$ is a field, if $d\not=0$ then this implies
  $u=1$ thus $d'=\sigma(d).$ If $d=0$ then $d'=\sigma(d)$ is trivially true, since $u\not=0$.
 The statement is now obvious.
 \\ (iii) is trivial.
 \end{proof}

We call  $$y_i^\prime = \tau^i(u)\tau^i(v)v^{-1}\sigma(y_i), \text{ for all } 0 \le i \le n - 1$$
 and
$$g'=L(u^{-1})\sigma g\sigma^{-1}L(uv)$$
\emph{critical relations}.

In view of Lemma \ref{le:2.8}, the subsets
$$ N^0(f) =\{i \in [0,n-1]\mid y_i = 0\},$$
$$N^1(f) =\{i \in [0,n-1]\mid y_i \in K \setminus (K^\times \cup \{0\})\},$$
$$N^2(f) =\{i \in [0,n-1]\mid y_i \in K^\times\} $$
for $[0,n-1] = \{0,1,\dots,n - 1\}$ are isomorphism invariants of $K^{(f,g)}$ that disjointly cover $[0,n-1]$, ditto for $N^j(g)$, $j = 0,1,2$.

 \begin{lemma}\label{le:iso=iso}
 Let $K\in {\rm Alg}(V)$ be a cyclic Galois field extension. 
 Suppose that $f = \sum_{i=0}^{n-1}L(y_i)\tau^i \in {\rm Gl}(V)$.
 Let
 $u,v\in K^\times$ and $\sigma\in {\rm Gal}(K/F)$, and put
 $$y_i'=\tau^i(u)\tau^i(v) v^{-1}\sigma(y_i)$$ for all $i\in\{0,\dots,n-1\}$. Define
 $f'= \sum_{i=0}^{n-1}L(y_i')\tau^i $. Then $f'\in {\rm Gl}(V)$.
  \end{lemma}

  \begin{proof}
   Put $g = \sigma = id_V$, $g^\prime = L(v)$ and $\varphi = L((uv)^{-1})$. Then the hypotheses  combined with Lemma \ref{le:2.8} show that
\[
\varphi : K^{(f,id_V)} \overset{\sim} \longrightarrow K^{(f^\prime,g^\prime)}
\]
is an isomorphism. By Theorem \ref{thm:1.12}, therefore $f^\prime = L(v^{-1})f L(uv)$ is bijective.
  \end{proof}

The first step towards a classification of algebras that have a cyclic Galois field extension of degree $n$ as unital heart is the following observation that lists the different types of  $(f,g)\in G_0(V)$.

 \begin{theorem}  \label{le:2.9fields}  
 Let $K\in {\rm Alg}(V)$ be a cyclic Galois field extension of degree $n$.
 Let $(f,g)\in G_0(V)$. Then there exist $(f',g')\in G_0(V) $ such that
 $K^{(f,g)}=K^{(f',g')}$ and $f'$ satisfies  exactly one of the following conditions:
 \begin{enumerate}
\item 
 $f'=id +L(y_1)\tau + L(y_2)\tau^2 + \dots + L(y_{n-1})\tau^{n-1}$
 such that $N^2(f') = [0,n-1]$
 and $N_{A_0/F}(1+y_1t+\dots+y_{n-1}t^{n-1})\not=0$.
 \item 
  $f'=id +L(y_1)\tau + L(y_2)\tau^2 + \dots + L(y_{n-1})\tau^{n-1}$, $N^0(f') $ is a nonempty set and
 $N_{A_0/F}(y)\not=0$.
 Every different subset $N^0(f') $ 
 yields a different subtype.
 \item 
 $f'=id $. 
 \item 
 $f'=L(y_1)\tau + L(y_2)\tau^2 + \dots + L(y_{n-1})\tau^{n-1}$  such that $N^2(f') = [0,n-1]$ 
 and $N_{A_0/F}(y)\not=0$.
 \item 
  $f'=L(y_1)\tau + L(y_2)\tau^2 + \dots + L(y_{n-1})\tau^{n-1}$ and $N^0(f') $ $ N^0(f')$ contains more than one element and
   $N_{A_0/F}(y)\not=0$.
 Every different subset $N^0(f') $ 
 yields a different subtype.
 
 \end{enumerate}
 \end{theorem}

 Two different subtypes in (2), respectively in (5), yield non-isomorphic Albert isotopes $K^{(f,g)}$.

 \begin{proof}
  By Theorem \ref{thm:2.2}, $f = \sum_{i=0}^{n-1}L(y_i)\tau^i$
 for some $y=y_0+y_1t+\dots+y_{n-1}t^{n-1}\in A_0$ such that $N_{A_0/F}(y)\not=0$.
 Moreover, if $y_0\in K^\times$ then there is $y_0^{-1}\in K^\times$ such that
 $1=y_0^{-1}y_0$.
 Applying the isomorphism criterium with $u=y_0^{-1}$, $\sigma=id$ we may assume $y_0=1$. 
 This leads us to two cases/types of isotopes, those with $y_0=1$ and those with $y_0=0$. The map $f'$ we construct in the following is an isomophism by Lemma \ref{le:iso=iso}.
 \\ Case $y_0=1$:
 \\
 (a) $y_i\in K^\times$ for all $i\in \{ 1,\dots,n-1\}$,
\\ (b) $ N^0(f') =\{i_1,\dots,i_s \}$ is non-empty.
\\ (c)  For all $i\geq 1$,  $y_i=0$.
 \\ Case $y_0=0$:
 \\ (a) $ N^2(f')$ has $n-1$ elements.
\\ (b) $ N^0(f')$ contains more than one element.
\end{proof}

This lists grows when $K\in {\rm Alg}(V)$ is, more generally, a Galois $C_n$-algebra.

\begin{proposition}  \label{le:2.9C_n} 
 Let $K\in {\rm Alg}(V)$ be a Galois $C_n$-algebra and $\tau$ e some distinguished automorphism of degree $n$.
 Let $(f,g)\in G$. Then there exist $(f',g')\in G_0(V)$ such that
 $K^{(f,g)}=K^{(f',g')}$ and $f'$  satisfies exactly one of the following conditions:
 \begin{enumerate}
  \item 
   $f'=id +L(y_1)\tau + L(y_2)\tau^2 + \dots + L(y_{n-1})\tau^{n-1}$ and  $N_{A_0/F}(y)\not=0$. Every $f'$ with different $N^0(f'), N^1(f'), N^2(f')$ yields a different subtype.

  \item 
  There is $y_0\in K\setminus K^\times$, $y_0\not=0$, such that $f'=L(y_0)+L(y_1)\tau + L(y_2)\tau^2 + \dots + L(y_{n-1})\tau^{n-1}$ and $N_{A_0/F}(y)\not=0$. Every $f'$ with  different $N^0(f'), N^1(f'), N^2(f')$ yields a different subtype.
  
 \item 
  $f'=L(y_1)\tau + L(y_2)\tau^2 + \dots + L(y_{n-1})\tau^{n-1}$ and $N_{A_0/F}(1+y_1t+\dots+y_{n-1}t^{n-1})\not=0$.
    Every $f'$ with  different  $N^0(f'), N^1(f'), N^2(f')$ yields a different subtype.

 \end{enumerate}
\end{proposition}

The types listed in  (1) to (5) in Theorem \ref{le:2.9fields} are contained in the above list  under the types in (1) and (3)
when $N^1(f')=\emptyset$. Two different subtypes of the same type yield non-isomorphic Albert isotopes $K^{(f,g)}$.

 \begin{proof}
  By Theorem \ref{thm:2.2}, we have  $f = \sum_{i=0}^{n-1}L(y_i)\tau^i$
 for some $y=y_0+y_1t+\dots+y_{n-1}t^{n-1}\in A_0$ such that $N_{A_0/F}(y)\not=0$.
 If $y_0\in K^\times$ then applying the isomorphism criterium with $u=y_0^{-1}$, $\sigma=id$ we may assume $y_0=1$. The remaining cases are $y_0=0$ and $y_0\in K\setminus K^\times$.

 We now look at the different cases for the other $y_i$'s. In each case, we distinguish the different possible ``partitions'' of
  $\{ 1,\dots,n-1\}$, allowing the sets $N^i(f')$ to be empty to cover all possible cases. In order to obtain a bijective $f'$, we require  that $N_{A_0/F}(y_0+y_1t+\dots+y_{n-1}t^{n-1})\not=0$ in each case.
 
 \end{proof}

\begin{corollary}  \label{le:2.9cubicfields} 
 Let $K\in {\rm Alg}(V)$ be a cubic Galois field extension.
 Let $(f,g)\in G_0(V) $. Then there exist $(f',g')\in G_0(V) $ such that
 $K^{(f,g)}=K^{(f',g')}$ and $f'$ satisfies exactly one of the following conditions:
 \begin{enumerate}
\item 
$f'=id +L(y_1)\tau + L(y_2)\tau^2$ for some $y_i\in K^\times$ with
  $ N_K(y_1) +  N_K(y_2) - T_{K}(\tau(y_1)\tau^2(y_2))\not=-1$.
 \item 
 $f'=id + L(y_2)\tau^2 $ 
 for some $y_2\in K^\times$ with  $ N_K(y_2)\not=-1$.
  \item 
   $f'=id +L(y_1)\tau $  
  for some $y_1\in K^\times$  with  $ N_K(y_1)\not=-1$.
 \item 
 $f'=id $. 
 \item 
 $f'=L(y_1)\tau + L(y_2)\tau^2 $ 
 for some $y_i\in K^\times$ with $N_K(y_2) \not=-  N_K(y_1) $.
 \item 
  $f'=L(y_2)\tau^2 $ 
  for some  $y_2\in K^\times$ with $N_K(y_2) \not=0$.
 \item 
 $f'=L(y_1)\tau $ 
 for $y_1\in K^\times$  with  $ N_K(y_1)  \not=0$.
 \end{enumerate}
 \end{corollary}

\begin{proof}
 By Lemma \ref{le:2.8}, the cases (1)$\--$(7) are mutually disjoint.
  Moreover, we know that
 $N_{A_0/F}(1+y_1t+y_{2}t^{2}) = N_K(y_0) +    N_K(y_1) +  N_K(y_2) - T_{K}(y_0\tau(y_1)\tau^2(y_2))$
 needs to be non-zero.
\end{proof}

\begin{corollary}  \label{le:2.9cubicetale} 
 Let $K\in {\rm Alg}(V)$ be a $C_3$-Galois algebra with a distinguished $\tau\in {\rm Aut}_F(K)$ of order 3, and write $N_{A_0}(y)=N_{A_0/F}(1+y_1t+y_{2}t^{2})$.
 Let $(f,g)\in G_0(V) $. Then there exist $(f',g')\in G_0(V) $ such that
 $K^{(f,g)}=K^{(f',g')}$ and $f'$ satisfies exactly one of the following conditions: (1) to (7) from Corollary \ref{le:2.9cubicfields},
\begin{enumerate}
 \setcounter{enumi}{7}
  \item 
   $f'=id +L(y_1)\tau + L(y_2)\tau^2$ for some $0\not= y_i\in K\setminus K^\times$ with $N_{A_0}(y)\not=0$.
  \item 
   $f'=id +L(y_1)\tau + L(y_2)\tau^2$ for some $0\not= y_1\in K\setminus K^\times$, $y_2\in K^\times$ with $N_{A_0}(y)\not=0$.
    \item 
     $f'=id +L(y_1)\tau + L(y_2)\tau^2$ for some  $y_1\in K^\times$,  $0\not= y_2\in K\setminus K^\times$ with $N_{A_0}(y)\not=0$.
 \item 
 $f'=L(y_1)\tau + L(y_2)\tau^2 $ 
 for some $0\not= y_i\in K\setminus K^\times$ with $N_{A_0}(y)\not=0$.
  \item 
    $f'=L(y_1)\tau + L(y_2)\tau^2$ for some $0\not= y_1\in K\setminus K^\times$, $y_2\in K^\times$ with $N_{A_0}(y)\not=0$.
    \item 
     $f'=L(y_1)\tau + L(y_2)\tau^2$ for some  $y_1\in K^\times$,  $0\not= y_2\in K\setminus K^\times$ with $N_{A_0}(y)\not=0$.
 \item 
  $f'=L(y_0)+L(y_1)\tau + L(y_2)\tau^2$ for some $0\not= y_0\in K\setminus K^\times$, $y_1,y_2\in K^\times$ with   $N_{A_0}(y)\not=0$.
 \item 
 $f'=L(y_0)+ L(y_2)\tau^2 $
 for some $y_2\in K^\times$ with $N_{A_0}(y_0+y_2t^2)\not=0$.
  \item 
   $f'=L(y_0)+L(y_1)\tau $
  for some $y_1\in K^\times$  with $N_{A_0}(y)\not=0$.
  \item 
    $f'=L(y_0)+ L(y_2)\tau^2 $
 for some $0\not= y_2\in K\setminus K^\times$ with $N_{A_0}(y)\not=0$. 
  \item 
   $f'=L(y_0)+L(y_1)\tau $
  for some $0\not= y_1\in K\setminus K^\times$  with $N_{A_0}(y)\not=0$.
 \item 
  $f'=L(y_0)+L(y_1)\tau + L(y_2)\tau^2$ for some $0\not= y_0\in K\setminus K^\times$, $0\not= y_i\in K\setminus K^\times$ with $N_{A_0}(y)\not=0$.
  \item 
    $f'=L(y_0)+L(y_1)\tau + L(y_2)\tau^2$ for some $0\not= y_0\in K\setminus K^\times$, $0\not= y_1\in K\setminus K^\times$, $y_2\in K^\times$ with $N_{A_0}(y)\not=0$.
    \item 
      $f'=L(y_0)+L(y_1)\tau + L(y_2)\tau^2$ for some $0\not= y_0\in K\setminus K^\times$, $y_1\in K^\times$,  $0\not= y_2\in K\setminus K^\times$ with $N_{A_0}(y)\not=0$.
      \end{enumerate}
 \end{corollary}

\begin{proof}
Cases (1) to (7) are proved analogously as when $K$ is a field. Additionally, we now also have to distinguish when $0\not= y_i\in K\setminus K^\times$ which leads to the other new cases: We note that $N_{A_0}(y)\not=0$ is equivalent to $y\in A_0$ being invertible, so that the condition that   $N_{A_0/F}(1+y_1t+y_{2}t^{2})\not=0$ excludes certain cases. By Lemma \ref{le:2.8}, the cases
(8)$\--$(20)  are also mutually disjoint.
\end{proof}

 Define $S(K)=\{x\in V\,|\, N(x)=1\}$. By Hilbert's Theorem 90, we know
$$ S(K)=\{ \frac{\tau(v)}{v}\,|\, v\in K^\times\}.$$
Choose a set of representatives $M$ of $K^\times/S(K)$ containing 1.

\begin{lemma}\label{le:M}
Let $K\in {\rm Alg}(V)$ be a cubic Galois extension. 
Let $y_i,y_i'\in M$ and $u,v\in K^\times$, $\sigma\in {\rm Gal}(K/F)$, such that
 $ y_i'=\tau^i(u)\tau^i(v) v^{-1}\sigma(y_i)$ for some $i\in\{1,\dots,n-1\}$ and $N(u)=1$. Then  $y_i=y_i'$.
\end{lemma}

\begin{proof}
Since $ y_i'=\tau^i(u)\tau^i(v) v^{-1}\sigma(y_i)$ we have $N(y_i')=N(y_i)$, thus $s:=y_i'y_i^{-1}\in S(K)$. The equation
 $ y_i=sy_i'$ means that $y_i,y_i'\in M$ are congruent modulo $S(K)$, thus $y_i=y_i'$.
\end{proof}

We will now use the set $M$ to attempt a classification of our algebras.

\ignore{
Since $N(x)=x\tau(x)\tau^2(x)\cdots \tau^{n-1}(x)$,
 we know $x^{-1}=N(x)^{-1}x\tau(x)\tau^2(x)\cdots  \tau^{n-1}(x)$  for $N(x)\not=0$.
}

\section{A classification of division algebras with  a cubic Galois field extension as the unital heart}\label{sec:4}

Let $A$ be a regular algebra over $F$ of dimension 3 with a unital heart isomorphic to a  cubic Galois field extension $K$ with ${\rm Gal}(K/F)=\langle \tau\rangle$. Then $A$ is a division algebra and
we know that $A\cong K^{(f,g)}$ for some $(f,g)\in G_0(V) $ by Theorem \ref{thm:1.10}. By Corollary \ref{le:2.9cubicfields} we are left with the following cases:
\\ \\ (1) $f=id+L(y_1)\tau + L(y_2)\tau^2$ for some $y_i\in K^\times$ such that
  $ N_K(y_1) +  N_K(y_2) - T_{K}(\tau(y_1)\tau^2(y_2))\not=-1$:
    Then there exists an element $y_1'\in M$, such that $y_1'=y_1 \,{\rm mod}\, S(K)$, i.e. there exists $v\in K^\times$ such that $y_1'=\tau(v)v^{-1}y_1$.
  Set $\sigma=id$, $u=1$ and $g'= g L(v)$ in Lemma \ref{le:2.8}. Then for
 $y_1'=\tau(v) v^{-1}y_1,$
  $y_2'=\tau^2(v) v^{-1}y_2$ we have
  $$K^{(1+L(y_1)\tau + L(y_2)\tau^2,g)}\cong K^{(1+L(y_1')\tau + L(y'_2)\tau^2,g')}.$$
  (2) $f=id + L(y_2)\tau^2 $  
 for some $y_2\in K^\times$ with  $ N_K(y_2)\not=-1$:
   Then there exists an element $y_2'\in M$, such that $y_2'=y_2 \,{\rm mod}\, S(K)$, i.e. there exists $v\in K^\times$ such that $y_2'=\tau^2(v)v^{-1}y_2$.
 Set $\sigma=id$ and $u=1$ in Lemma \ref{le:2.8}. Then choose $g'= g L(v)$ and for
 $
  y_2'=\tau^2(v) v^{-1}y_2$
 we have
  $$K^{(1 + L(y_2)\tau^2,g)}\cong K^{(1+ L(y'_2)\tau^2,g')}.$$
 (3) $f=id +L(y_1)\tau $  
  for some $y_1\in K^\times$  with  $ N_K(y_1)\not=-1$:
  Again there exists an element $y_1'\in M$, such that  $N_K(y_1)=N_K(y_1')$, and there exists $v\in K^\times$ such that $y_1'=\tau(v)v^{-1}y_1$.
Set $\sigma=id$, $u=1$ and $g'= g L(v)$ in Lemma \ref{le:2.8}. Then for
 $y_1'=\tau(v) v^{-1}y_1 $ 
 we have
  $$K^{(1 + L(y_1)\tau,g)}\cong K^{(1+ L(y'_1)\tau,g')}.$$
  (4) $f=id $: 
Then
$g=L(x_0)+L(x_1)\tau + L(x_2)\tau^2$ for $x_i\in K$ with $N_{A_0}(x_0+x_1t+x_2t^2)\not=0$.
We have two cases to consider: $x_0=0$, then choose $v=1$;
 $x_0\in K^\times$, 
then choose $v=x_0^{-1}$. Apply Lemma \ref{le:2.8} for $u=1$, $v$, $\sigma=id$ to see that for
 $g'= g L(v)$ we obtain  $K^{(1,g)}\cong K^{(1,g')}$. Now $g'=L(e)+L(x_1)\tau + L(x_2)\tau^2$ for $e,x_i\in K$ with $N_{A_0}(e+x_1t+x_2t^2)\not=0$ and $e\in \{0,1\}$.
 \\ (4.i) If $e=0$ then $g'=L(x_1)\tau + L(x_2)\tau^2$ for $x_i\in K$ with $N_{A_0}(x_1t+x_2t^2)\not=0$.
  \\ (4.ii)  If $e=1$ then $g'=1+L(x_1)\tau + L(x_2)\tau^2$ for $x_i\in K$ with $N_{A_0}(1+x_1t+x_2t^2)\not=0$.
\\\\ (5) $f=L(y_1)\tau + L(y_2)\tau^2 $ 
 for some $y_i\in K^\times$ with $N_K(y_2) \not=-  N_K(y_1) $:
  There exists an element $y_1'\in M$, such that $N_K(y_1)=N_K(y_1')$ and $v\in K^\times$ such that $y_1'=\tau(v)v^{-1}y_1$.
 Set $\sigma=id$ and $u=1$ in Lemma \ref{le:2.8}. Then for $g'= g L(v)$ and for
  $y_2'=\tau^2(v_1) v_1^{-1}y_2$
 we get
  $$K^{(L(y_1)\tau + L(y_2)\tau^2,g)}\cong K^{(L(y_1')\tau + L(y'_2)\tau^2,g')}$$
  by Lemma \ref{le:2.8}.
 \\\\ (6)  $f=L(y_2)\tau^2 $ 
  for some  $y_2\in K^\times$ with $N_K(y_2) \not=0$:
  Again there exists an element $y_2'\in M$, such that $N_K(y_2)=N_K(y_2')$
  and $v\in K^\times$ such that $y_2'=\tau^2(v)v^{-1}y_2$.
  Set $\sigma=id$ and $u=1$ in Lemma \ref{le:2.8}. Then for $g'= g L(v)$
 we obtain
  $$K^{( L(y_2)\tau^2,g)}\cong K^{( L(y'_2)\tau^2,g')}.$$
  (7)  $f=L(y_1)\tau $ 
 for $y_1\in K^\times$  with  $ N_K(y_1)  \not=0$: Analogously as above, we use Lemma \ref{le:2.8} for $\sigma=id$, $u=1$, and $v$ with
 $y_1'=\tau(v)v^{-1}y_1$. We get
  $ K^{(L(y_1)\tau,g)}$ for $y_1\in M$,  such that $ N_K(y_1) \not=0$ and  $g\in {\rm Gl(V)}$. We obtain:

 \begin{theorem} (Tight Enumeration Theorem for cubic unital hearts)\label{thm:tightenumcubicext} 
 Let $A$ be a division algebra of dimension three with unital heart isomorphic to a  cubic Galois field extension, ${\rm Gal}(K/F)=\langle \tau\rangle$.
 Then $A$ is isomorphic to precisely one of the following types (in the following, $g\in {\rm Gl(V)}$):
 \begin{enumerate}
 \item 
  $K^{(1+L(y_1)\tau + L(y_2)\tau^2,g)}$
 for $y_1\in M$, $y_2\in K^\times$ with
  $ N_K(y_1) +  N_K(y_2) - T_{K}(\tau(y_1)\tau^2(y_2))\not=-1$.
 \item 
 $ K^{(1 + L(y_2)\tau^2,g)}$
 for $y_2\in M$, such that  $ N_K(y_2)\not=-1$.
\item 
$ K^{(1 + L(y_1)\tau,g)}$
 for $y_1\in M$, such that  $ N_K(y_1)\not=-1$.
\item 
$K^{(id,g)}$ with $g=L(x_1)\tau + L(x_2)\tau^2$ for $x_i\in K$ with $N_{A_0}(x_1t+x_2t^2)\not=0$.
\item 
$K^{(id,g)}$ with  $g=1+L(x_1)\tau + L(x_2)\tau^2$ for $x_i\in K$ with $N_{A_0}(1+x_1t+x_2t^2)\not=0$.
 \item 
  $K^{(L(y_1)\tau + L(y_2)\tau^2,g)}$
 for $y_1\in M$, $y_2\in K^\times$, such that $N_K(y_2) \not=-  N_K(y_1) $.
\item 
$ K^{(L(y_2)\tau^2,g)}$ for $y_2\in M$  with $ N_K(y_2) \not=0$.
 \item 
 $ K^{(L(y_1)\tau,g)}$ for $y_1\in M$  with $ N_K(y_1) \not=0$.
 \end{enumerate}
 \end{theorem}

 \begin{proof}
 By Lemma \ref{le:2.8}, two algebras $A$ and $B$ that belong to different types (1), (2), (3) and (6), (7), (8)  cannot be isomorphic. The same holds for algebras that lie in either (4) or (5), again these are not isomorphic to any of the types listed in (1), (2), (3) or (6), (7), (8).
  Since $(A^{(f,g)})^{op}\cong (A^{op})^{(g,f)}$ we can apply  Lemma \ref{le:2.8} to the opposite algebras as well which shows us that algebras of different type (4) and (5) also are not isomorphic.
 \end{proof}

  To complete the classification of three-dimensional algebras which have a cubic Galois field extension as unital heart, we now check when two algebras of the same type are isomorphic:

 \begin{theorem} (Classification Theorem for algebras with a cubic Galois field extension as unital heart) \label{thm:classcubic}
 Let $K/F$ be a  cubic Galois field extension, ${\rm Gal}(K/F)=\langle \tau\rangle$. 
 \begin{enumerate}
 \item 
 For $y_1,z_1\in M$, $y_2,z_2\in K^\times$ with
  $ N_K(y_1) +  N_K(y_2) - T_{K}(\tau(y_1)\tau^2(y_2))\not=-1$, $ N_K(z_1) +  N_K(z_2) - T_{K}(\tau(z_1)\tau^2(z_2))\not=-1$  and $g,g'\in {\rm Gl(V)}$, we have
  $$K^{(L(d)+L(y_1)\tau + L(y_2)\tau^2,g)}\cong K^{(L(d')+L(y')\tau + L(y'_2)\tau^2,g')}$$
 if and only if $z_1=y_1$ and exactly on of the following holds:
  \\ (i) $z_2=y_2$ and there exists $a\in F^\times$  such that $ g'=ag,$
  \\ (ii) $ z_2=\tau^2(y_1^{-1}) y_1 \tau(y_2)$  and there exists $ a\in F^\times $ such that $  g'=a \tau g\tau^{-1}L(y_1)^{-1}$,
 \\ (iii) $z_2 = \tau(y_1)\tau^2(y_1^{-1})\tau^2(y_2)$ and there exists $a \in F^\times$ such that
\[
g^\prime  = a\tau^2gL(y_1)\tau^{-2}.
\]
 \item 
 For $y_2,z_2\in M$, such that  $ N_K(y_2)\not=-1$, $ N_K(z_2)\not=-1$ and $g,g'\in {\rm Gl(V)}$, we have
 $$ K^{(1 + L(y_2)\tau^2,g)}\cong K^{(1 + L(z_2)\tau^2,g')}$$
  if and only if $z_2=y_2$ and exactly on of the following holds:
  \\ (i) there exists $ a\in F^\times $  such that  $g'=ag,$
  \\ (ii)  there exists $ a\in F^\times $  such that $v = a\tau(y_2^{-1})$ and $g'= a \tau g\tau^{-1}L(\tau(y_2^{-1}))$
    \\ (iii) there exists $ a\in K^\times $  such that $v=ay_2^{-1}$ and   $g'=a \tau^2 g\tau^{-2}L(y_2)^{-1}$.
   
  \item 
  For $y_1,z_1\in M$, such that  $ N_K(y_1)\not=-1$, $ N_K(z_1)\not=-1$ and $g,g'\in {\rm Gl(V)}$, we have
   $$ K^{(1 + L(y_1)\tau,g)}\cong K^{(1 + L(z_1)\tau,g)}$$
  if and only if $z_1=y_1$ and exactly on of the following holds:
  \\ (i) there exists $ a\in F^\times$ such that $g'=ag,$
  \\ (ii) there is $a\in F^\times$ such that  $v=ay_1^{-1}$ and  $g'=a \tau g\tau^{-1}L(y_1)^{-1}$.
 \\ (iii) there is $a\in F^\times$ such that $v = a\tau^2(y_1)$  and
$g^\prime = a\tau^2gL(y_1)\tau^{-2},$

 \item 
   For $x_i,w_i\in K$ with $N_{A_0}(x_1t+x_2t^2)\not=0$, $N_{A_0}(w_1t+w_2t^2)\not=0$, we have
   $$K^{(id,L(x_1)\tau + L(x_2)\tau^2)}\cong K^{(id,L(w_1)\tau + L(w_2)\tau^2)}$$
   if and only if there exists some $\sigma \in {\rm Gal}(K/F)$ where $w_2 = \tau(w_1)\sigma(\tau(x_1^{-1}))\sigma(x_2)$,

\item 
 For $x_i,w_i\in K$ with $N_{A_0}(1+x_1t+x_2t^2)\not=0$, $N_{A_0}(1+w_1t+w_2t^2)\not=0$ we have
 $$K^{(id,1+L(x_1)\tau + L(x_2)\tau^2)}\cong K^{(id,1+L(w_1)\tau + L(w_2)\tau^2)}$$
 if and only if there exists  $\sigma\in \{id, \tau,\tau^2\}$,  such that $ w_1= \sigma(x_1)$ and $ w_2= \sigma(x_2)$.
 \item 

 For $y_2,z_2 \in M$ such that  $N_K(y_2) \not=- 1 $ and $g,g^\prime \in {\rm Gl}(V)$,
\[
K^{(\tau + L(y_2)\tau^2,g)} \cong K^{(\tau + L(z_2)\tau^2,g^\prime)}
\]
if and only if $y_2 = z_2$ and there are $v \in K^\times$, $\sigma \in {\rm Gal}(K/F)$, such that
$$u = \tau^2(v)v^{-1}, \quad y_2 = \tau(v)v^{-1 }\sigma(y_2), \quad g^\prime = L(u^{-1})\sigma g\sigma^{-1}L(uv).$$
For $\sigma = id,\tau,\tau^2$ we obtain $v = a,ay_2^{-1},a\tau^2(y_2)$, respectively, for some $a \in F^\times$.
\item 

For  $g,g^\prime \in {\rm Gl}(V)$ and $i\in \{1,2\}$, we have $K^{(\tau^i,g)} \cong K^{(\tau^i,g^\prime)}$ if and only if there exist $v \in K^\times$, and $\sigma \in  {\rm Gal}(K/F)$ such that
\[
g' = L( \tau^{3-i}(v^{-1})v)\sigma g\sigma^{-1}L( \tau^{3-i}(v)).
\]
\end{enumerate}
 \end{theorem}

 \begin{proof}
By Lemma \ref{le:2.8}, two algebras mentioned in any of the cases of Theorem \ref{le:2.9cubicfields} are isomorphic if and only if there are $u,v\in K^\times$ and $\sigma\in {\rm Gal}(K/F)$, such that the criteria of Lemma \ref{le:2.8} are satisfied. We call them \emph{critical relations} following \cite{P}. Here, $\sigma\in \{id,\tau,\tau^2\}$. By Lemma \ref{le:M}, if $z_i,y_i\in M$ and
$z_i=\tau^i(v)\tau^i(v)  v^{-1}\sigma(y_i)$ then $z_i=y_i$. We go through the cases of Theorem \ref{le:2.9cubicfields}:
 \\\\ (1)
  Since we have $u=1$ here, the critical relations are  $z_1=\tau(v) v^{-1}\sigma(y_1)$, $z_2=\tau^2(v) v^{-1}\sigma(y_2)$ and $g'=\sigma g\sigma^{-1}L(v)$, which implies
   $z_1=y_1$ since $z_1,y_1\in M$.

    Assume $\sigma=id$. Using $z_1=y_1$ and  $y_1=\tau(v) v^{-1}\sigma(y_1)$ this yields $\tau(v) v^{-1}=1$, that means
$\tau(v)=v$, so $v=a\in F^\times$. Therefore the critical relations here become 
$z_1=y_1$,  $z_2=y_2$ and there exists $a\in F^\times$  such that $g'=ag$.

 Assume $\sigma=\tau$, i.e. $y_1=\tau(v) v^{-1}\tau(y_1)$ which is the same as $\tau(vy_1)=vy_1$, i.e.
 $v y_1\in F^\times$.
 Now the critical relations become: there is $a\in F^\times$,   $v=ay_1^{-1}$ and 
  $z_2=\tau^2(y_1^{-1}) y_1 \tau(y_2)$,  $g'=a \tau g\tau^{-1}L(y_1)^{-1}$.

  Assume $\sigma=\tau^2$, i.e. the critical relations are
   $y_1 = \tau(v)v^{-1}\tau^2(y_1)$  and $g'= \tau^2 g\tau^{-2} L(v)$.  In order to deal with this version, we will use the formula
\begin{align}
\label{SQUINV} \tau^2(x)x^{-1} = \tau\big(\tau(x)x\big)\big(\tau(x)x\big)^{-1} &&(x \in K^\times)
\end{align}
and obtain the following chain of equivalent conditions:
\begin{align*}
y_1 = \tau(v)v^{-1}\tau^2(y_1) \Longleftrightarrow\,\,&\tau(v)v^{-1} = \tau\big(\tau(y_1^{-1})y_1^{-1}\big)\big(\tau(y_1^{-1})y_1^{-1}\big)^{-1} \\
\Longleftrightarrow\,\,&\exists a \in F^\times\,:\;v = a\tau(y_1^{-1})y_1^{-1} = aN_K(y_1)^{-1}\tau^2(y_1) \\
\Longleftrightarrow\,\,&\exists a \in F^\times\,:\; v = a\tau^2(y_1).
\end{align*}
Hence there is $a\in F^\times$, such that $v = a\tau^2(y_1)$ and
$z_2 = \tau(y_1)\tau^2(y_1^{-1})\tau^2(y_2)$ and
$g^\prime = a\tau^2g\tau^{-2}L\big(\tau^2(y_1)\big) = a\tau^2gL(y_1)\tau^{-2}.$
 \\\\ (2)
 Since  $u=1$, here the critical relations are $z_2=\tau^2(v) v^{-1}\sigma(y_2)$ and $g'=\sigma g\sigma^{-1}L(v)$,  which implies
  $z_2=y_2$ by the definition of $M$.

 Assume $\sigma=id$. Using $z_2=y_2$, the equation
 $z_2=\tau^2(v) v^{-1}\sigma(y_2)$  yields $\tau^2(v) v^{-1}=1$, that means $\tau^2(v)=v$, so $v=a\in F^\times$.
 Therefore the critical relations are $z_2=y_2$ and
  there exists $a\in F^\times$ such that $v=a$ and  $g'=a g$.

   Assume $\sigma=\tau$. Then the critical relations imply $y_2=\tau^2(v) v^{-1}\tau(y_2)$ and $g'=\tau g\tau^{-1}L(v)$.
    and \eqref{SQUINV} can be applied. We obtain the following chain of equivalent conditions:
\begin{align*}
y_2 = \tau^2(v)v^{-1}\tau^2(y_2) \Longleftrightarrow\,\,&\tau(y_2^{-1})y_2 = \tau\big(\tau(v)v\big)\big(\tau(v)v\big)^{-1} \\
\Longleftrightarrow\,\,&\exists a \in F^\times\,:\;y_2 = a\tau(v)v = aN_K(v^{-1})^{-1}\tau^2(v^{-1}) \\
\Longleftrightarrow\,\,&\exists a \in F^\times\,:\; y_2 = a\tau^2(v^{-1})\\
\Longleftrightarrow\,\,&\exists a \in F^\times\,:\; v = a\tau(y_2^{-1}).
\end{align*}
Thus $g'=\tau g\tau^{-1}L(a\tau(y_2^{-1}))= a \tau g\tau^{-1}L(\tau(y_2^{-1}))$.

Assume $\sigma=\tau^2$.   Then the critical relations are $y_2=\tau^2(v) v^{-1}\tau^2(y_2)$ and $g'=\tau^2 g\tau^{-2}L(v)$.
    The first equation is the same as $\tau^2(vy_2)=vy_2$, i.e. $v y_2\in F^\times$, denote it by  $a=v y_2$.
 Hence there is $a\in F^\times$,  such that $v=ay_2^{-1}$ and   $g'=a \tau^2 g\tau^{-2}L(y_2)^{-1}$.
 \\\\ (3) Here the critical relations imply that 
  $z_1=\tau(v) v^{-1}\sigma(y_1)$ and $g'=\sigma g\sigma^{-1}L(v)$, hence $z_1=y_1$ by the definition of $M$.

  Assume $\sigma=id$. This yields $\tau(v) v^{-1}=1$, that means $\tau(v)=v$, so $v=a\in F^\times$.
 Therefore the critical relations are $z_1=y_1$ and there exists $a\in F^\times$ such that $g'=ag$.

 Assume $\sigma=\tau$, i.e. $y_1=\tau(v) v^{-1}\tau(y_1)$.
 Then there is $a\in F^\times$,   $v=ay_1^{-1}$ and  $g'=a \tau g\tau^{-1}L(y_1)^{-1}$.

  Assume $\sigma=\tau^2$,   i.e.  $y_1=\tau(v) v^{-1}\tau^2(y_1)$. Then analogously as in (1) (iii), we obtain that there is $a\in F^\times$, such that $v = a\tau^2(y_1)$  and
$g^\prime = a\tau^2g\tau^{-2}L\big(\tau^2(y_1)\big) = a\tau^2gL(y_1)\tau^{-2}.$
  \\\\ (4) $K^{(1,g)}$ with $g=L(x_1)\tau + L(x_2)\tau^2$ for $x_i\in K$ with $N_{A_0}(x_1t+x_2t^2)\not=0$.
Here, the critical relations reduce to $u=1$ and
$$L(w_1)\tau + L(w_2)\tau^2=\sigma ( L(x_1)\tau + L(x_2)\tau^2  )\sigma^{-1}L(v)$$
 $$ =  ( L(\sigma(x_1))\tau L(v) +  L(\sigma(x_2))\tau^2   L(v)$$
$$ = L( \sigma(x_1) \tau(v)) \tau + L(\sigma(x_2)\tau^2 (v) )\tau^2.$$
By Corollary \ref{cor:2.9}, this is equivalent to
$w_1=\sigma(x_1) \tau(v)$ and $ w_2= \sigma(x_2)\tau^2 (v).$
Replacing $v$ by $\tau^2(v)$ if necessary, this condition that is necessary and sufficient for the two algebras to be isomorphic may also be expressed as follows: there exist $v \in K^\times$, $\sigma \in {\rm Gal}(K/F)$ having $w_i = \tau^{i-1}(v)\sigma(x_i)$ for $i = 1,2$. This implies $v = w_1\sigma(x_1^{-1})$ and $w_2 = \tau(w_1)\sigma(\tau(x_1^{-1}))\sigma(x_2)$. Hence the two algebras are isomorphic if and only if some $\sigma \in {\rm Gal}(K/F)$ has $w_2 = \tau(w_1)\sigma(\tau(x_1^{-1}))\sigma(x_2)$.
\\\\ (5) 
Here, the critical relations reduce to $u=1$ and
$$1+L(w_1)\tau + L(w_2)\tau^2 =  \sigma ( 1+L(x_1)\tau + L(x_2)\tau^2  )\sigma^{-1}L(v)$$
$$ = L(v)+  \sigma L(x_1)\tau \sigma^{-1}L(v) + \sigma L(x_2)\tau^2  \sigma^{-1}L(v)$$
$$ = L(v)+  L( \sigma(x_1) \tau(v)) \tau + L(\sigma(x_2)\tau^2 (v) )\tau^2.$$
 By Corollary \ref{cor:2.9}, this is equivalent to
$$ L(v)=1, w_1= \sigma(x_1) \tau(v), w_2= \sigma(x_2)\tau^2 (v).$$
This is equivalent to $v=1$, $ w_1= \sigma(x_1)$, and $w_2= \sigma(x_2).$
 \\\\ (6)
Let $y_i,z_i \in K^\times$ for $i = 1,2$ such that $N_K(y_1) \ne -N_K(y_2)$, $N_K(z_1) \ne -N_K(z_2)$, and let $g,g^\prime \in {\rm Gl}(V)$. By Theorem \ref{thm:1.12},
\[
K^{(L(y_1)\tau + L(y_2)\tau^2,g)} \cong K^{(L(z_1)\tau + L(z_2)\tau^2,g^\prime)}
\]
if and only if there are $u,v \in K^\times$, $\sigma \in {\rm Gal}(K/F)$ satisfying
\begin{align}
\label{ZETTAU} z_i = \tau^i(u)\tau^i(v) v^{-1}\sigma(y_i), \quad g^\prime = L(u^{-1})\sigma g\sigma^{-1}L(uv) &&(i = 1,2).
\end{align}
For $u = \tau^2(y_1^{-1})$, $v = 1_K$, $\sigma = 1_V$, this implies $z_1 = 1_K$. Thus the algebras in (6), up to isomorphism, have the form $K^{(\tau + L(y_2)\tau^2,g)}$, $y_2 \in K^\times$, $g \in {\rm Gl}(V)$. From now on, we may therefore assume $y_1 = z_1 = 1_K$: Then Equation \eqref{ZETTAU} amounts to
\begin{align}
\label{UTAUVE} u = \tau^2(v)v^{-1}, \quad z_2 = \tau(v)v^{-1}\sigma(y_2), \quad g^\prime = L(u^{-1})\sigma g \sigma^{-1}L(uv).
\end{align}
The second equation of \eqref{UTAUVE} for $\sigma = 1_V$ and an appropriate choice  of $v$ implies $z_2 \in M$. Thus the algebras in case (6) up to isomorphism have the form $K^{(\tau + L(y_2)\tau^2,g)}$, for some $y_2 \in M$, $g \in {\rm Gl}(V)$. In the preceding discussion, we may therefore assume $y_2,z_2 \in M$, whence Equation \eqref{UTAUVE} implies $y_2 = z_2$ and thus becomes equivalent to
\begin{align}
\label{UTAUEQ} u = \tau^2(v)v^{-1}, \quad y_2 = \tau(v)v^{-1 }\sigma(y_2), \quad g^\prime = L(u^{-1})\sigma g\sigma^{-1}L(uv).
\end{align}
Summing up, therefore,
the algebras in case (6) are isomorphic to $K^{(\tau + L(y_2)\tau^2,g)}$ for some $y_2 \in M$, $g \in {\rm Gl}(V)$, and given $y_2,z_2 \in M$, $g,g^\prime \in {\rm Gl}(V)$,
\[
K^{(\tau + L(y_2)\tau^2,g)} \cong K^{(\tau + L(z_2)\tau^2,g^\prime)}
\]
if and only if $y_2 = z_2$ and Equation \eqref{UTAUEQ} holds for some $v \in K^\times$, $\sigma \in {\rm Gal}(K/F)$.

For $\sigma = id,\tau,\tau^2$, respectively, the usual arguments employed already in the proof of case (1) yield $v = a,ay_2^{-1},a\tau^2(y_2)$, respectively, for some $a \in F^\times$. Details are left to the reader.
 \\ \\
(7) This case deals with algebras of the form
$ K^{(L(y_i)\tau^i,g)}$, with $y_i \in K^\times$, $ g \in {\rm Gl}(V)$, $ i \in \{1,2\}$.
By Theorem \ref{thm:1.12},  we may assume $y_i = 1_K$ and  we know that, given $g,g^\prime \in {\rm Gl}(V)$, we have $K^{(\tau^i,g)} \cong K^{(\tau^i,g^\prime)}$ if and only if there exist $v \in K^\times$, $\sigma \in  {\rm Gal}(K/F)$ satisfying
\[
g^\prime = L(u^{-1})\sigma g\sigma^{-1}L(uv), \quad u = \tau^{3-i}(v)v^{-1}.
\]
We obtain $g^\prime = L( \tau^{3-i}(v^{-1})v)\sigma g\sigma^{-1}L( \tau^{3-i}(v))$.

\end{proof}

We observe that for $g=L(x_0)+L(x_1)\tau + L(x_2)\tau^2$, $g'=L(w_0)+L(w_1)\tau + L(w_2)\tau^2$, in cases (1), (2), (3),
$$g'=a\sigma g \sigma^{-1}=a (L(\sigma(x_0))  + L(\sigma(x_1))\tau  + L(\sigma(x_2)) \tau^2 )$$
 if and only if
$$w_i=a\sigma(x_i) \text{ for } i\in \{0,1,2\}$$
and that
$$g'=L(u^{-1})\sigma g \sigma^{-1}=a (L(\sigma(x_0))  + L(\sigma(x_1))\tau  + L(\sigma(x_2)) \tau^2 )L(uv)$$
 if and only if
$$w_i=u^{-1}\sigma(x_i)\tau^i(u)\tau^i(v).$$

\section{Towards a classifications of algebras whose unital heart is a cyclic Galois field extension}\label{sec5}

Let $A$ be a division algebra over $F$ of dimension $n$ with a unital heart isomorphic to a cyclic Galois field extension $K$, ${\rm Gal}(K/F)=\langle \tau\rangle$.  The methology to "sort'' the algebras with unital hear $K$ is the same as demonstrated in the cubic field case:
by Theorem \ref{thm:1.10},  $A\cong K^{(f,g)}$ for some $(f,g)\in G_0(V) $, and by Theorem \ref{le:2.9fields} we are left with a list of cases, some of which we can refine. Write
$f'=L(e)+L(y_1)\tau + L(y_2)\tau^2 + \dots + L(y_{n-1})\tau^{n-1}$, $e\in \{1,0\}$.
When $e=1$ and there exists $y_i\in K^\times$ in the expression of $f'$ such that $\tau^{i}$ generates $G$, choose the smallest such $y_i$ and scale it to an element in $M$.
\\ \\ (1) $f'=id +L(y_1)\tau + L(y_2)\tau^2 + \dots + L(y_{n-1})\tau^{n-1}$, $N^0(f')$ is empty and
 $N_{A_0/F}(y)=N_{A_0/F}(1+y_1t+\dots+y_{n-1}t^{n-1})\not=0$:
  Then there exists $y_1'\in M$, such that $N(y_1)=N(y_1')$ and $v\in K^\times$ with $y_1'=\tau(v)v^{-1}y_1$. Set $\sigma=id$ and $u=1$ in Lemma \ref{le:2.8}. Then put $g'= g L(v)$ and for $y_{i}'=\tau^{i}(u)\tau^{i}(v) v^{-1}y_{i}$
 we have
 $$K^{(1+L(y_1)\tau + L(y_2)\tau^2 + \dots + L(y_{n-1})\tau^{n-1},g)}\cong K^{(1+L(y')\tau + L(y'_2)\tau^2 + \dots + L(y'_{n-1})\tau^{n-1},g')}$$
  (2) $f'=id +L(y_1)\tau + L(y_2)\tau^2 + \dots + L(y_{n-1})\tau^{n-1}$, $N^0(f')$ is not empty, and $N_{A_0/F}(1+y_1t+\dots+y_{n-1}t^{n-1})\not=0$:
 \\ Case 1: there exists  $i_k\in N^2(f')$
 such that $\tau^{i_k}$ generates ${\rm Gal}(K/F)$.
 Choose the smallest such $i_k$, that means $y_{i_k}$ is nonzero. 
    By Hilbert's Theorem 90, $S(K)=\{\tau^{i_k}(v)v^{-1},|\, v\in K^\times\}$.  Thus there exists $y_{i_k}'\in M$ and $v\in K^\times$, such that $y_{i_k}'=\tau^{i_k}(v)v^{-1}y_{i_k}$. Set $\sigma=id$ and $u=1$ in Lemma \ref{le:2.8}. For $g'= g L(v)$ and for
  $$y_1'=\tau^2(v) v^{-1}y_1, \dots, y_{i_k}'=\tau(v) v^{-1}y_{i_k}, \dots, y_{n-1}'=\tau^{n-1}(v) v^{-1}y_{n-1}$$
 we have
 $$K^{(1+L(y_1)\tau + L(y_2)\tau^2 + \dots + L(y_{n-1})\tau^{n-1},g)}\cong K^{(1+L(y')\tau + L(y'_2)\tau^2 + \dots + L(y'_{n-1})\tau^{n-1},g')}$$
 This implies that
 $$K\cong K^{(1+L(y_1)\tau + L(y_2)\tau^2 + \dots + L(y_{n-1})\tau^{n-1},g)}$$
 for $y_{i_k}\in M$, and for $y_i\in K$, such that for some nonempty set $\{i_1,\dots,i_s \}\subset  \{ 1,\dots,n-1\}$ of less than $n-1$ elements, the $y_{i_1},\dots,y_{i_s}$ equal 0,
  $ N_{A_0}(1+y_1t+\dots+y_{n-1}t^{n-1})\not=0$, and  $g\in {\rm Gl(V)}$.
 \\ Case 2:  For all $i_k\in N^2(f')$, $\tau^{i_k}$ does not generate ${\rm Gal}(K/F)$.
 Then we cannot ``scale'' any entry $y_i$ so that it lies in $M$. This case can only happen when $n$ is not prime.
  \\\\ (3) $f'=id $: 
$g=L(x_0)+L(x_1)\tau  + \dots + L(x_{n-1})\tau^{n-1}$ for $x_i\in K$ with $N_{A_0}(x_0+x_1t+\dots +x_{n-1}t^{n-1})\not=0$.
We have two cases to consider: $x_0=0$, then choose $v=1$, 
 and the case $x_0\in K^\times$, 
then choose $v=x_0^{-1}$. Apply Lemma \ref{le:2.8} for $u=1$, $v$, $\sigma=id$ to see that for
 $g'= g L(v)$ we obtain  $K^{(1,g)}\cong K^{(1,g')}$. Now $g'=(L(e)+L(x_1)\tau +  \dots + L(x_{n-1})\tau^{n-1}$ for $e,x_i\in K$ with $N_{A_0}(e+x_1t+\dots+x_{n-1}t^{n-1})\not=0$ and $e\in \{0,1\}$.
 \\
If $e=0$ then $g'=L(x_1)\tau + \dots+x_{n-1}t^{n-1}$ for $x_i\in K$ with $N_{A_0}(x_1t+\dots+x_{n-1}t^{n-1})\not=0$.
\\
If $e=1$ then $g'=1+L(x_1)\tau + \dots+x_{n-1}t^{n-1}$ for $x_i\in K$ with $N_{A_0}(1+x_1t+\dots+x_{n-1}t^{n-1})\not=0$.
\\\\ (4) $f'=L(y_1)\tau + L(y_2)\tau^2 + \dots + L(y_{n-1})\tau^{n-1}$,  $N^0(f')$ is empty and $N_{A_0/F}(y_1t+\dots+y_{n-1}t^{n-1})\not=0$.
 \\\\ (5) $f'=L(y_1)\tau + L(y_2)\tau^2 + \dots + L(y_{n-1})\tau^{n-1}$,  $N^0(f')$ is not empty, and   $N_{A_0/F}(y_1t+\dots+y_{n-1}t^{n-1})\not=0$: These cases we cannot refine any more.

By Proposition \ref{le:2.8}, two algebras $A$ and $B$ that belong to different types (1), (2) and different subsets $N^0(f')$,  and (5) with different subsets $N^0(f')$, cannot be isomorphic. The same holds for algebras that lie in either (3) or (4), again these are not isomorphic to any of the types listed in (1), (2, $N^0(f')$)  and (5, $N^0(f')$).
Since $(A^{(f,g)})^{op}\cong (A^{op})^{(g,f)}$ we can apply Proposition \ref{le:2.8} to the opposite algebras as well which shows us that algebras of type (3) and (4) also are not isomorphic. We proved:

\begin{theorem} (Enumeration Theorem for  unital hearts that are a cyclic field extension)\label{thm:tightenumcubicext} 
Let $A$ be a division algebra of dimension $n$ with unital heart isomorphic to a  cyclic Galois field extension $K/F$ with ${\rm Gal}(K/F)=\langle \tau\rangle$. Then $A$ is isomorphic to precisely one of the following:
\begin{enumerate}
 \item 
 $K^{(1+L(y_1)\tau + L(y_2)\tau^2 + \dots + L(y_{n-1})\tau^{n-1},g)}$ such that $N^0(f')$ is empty,
  $y_1\in M$, $N_{A_0}(1+y_1t+\dots+y_{n-1}t^{n-1})\not=0$ and  $g\in {\rm Gl}(V)$.
 \item 
 $K^{(1+L(y_1)\tau + L(y_2)\tau^2 + \dots + L(y_{n-1})\tau^{n-1},g)}$ such that  $N^0(f')$ is not empty,
   $N_{A_0/F}(1+y_1t+\dots+y_{n-1}t^{n-1})\not=0$, and $g\in {\rm Gl(V)}$.
    \\ If  $\tau^{i_k}$ does not generate ${\rm Gal}(K/F)$ for all $i_k\in N^2(f')$, then we cannot ``scale'' any nonzero entry  $y_i$ so that it lies in $M$.
  \\ If there exists  $i_k\in N^2(f')$
 such that $\tau^{i_k}$ generates ${\rm Gal}(K/F)$, then
  take the smallest such  $y_{i_k}\in M$, and get
  $K\cong K^{(1+L(y_1)\tau + L(y_2)\tau^2 + \dots + L(y_{n-1})\tau^{n-1},g)}$
 for $y_{i_k}\in M$.
 \\ Each set $N^0(f')$ yields a different type.
  \item 
  $K^{(1,g)}$ for $g=L(x_1)\tau +  \dots + L(x_{n-1})\tau^{n-1}$ for $x_i\in K$ with $N_{A_0}(x_1t+\dots+x_{n-1}t^{n-1})\not=0$.
 \item 
  $K^{(1,g)}$ for  $g=1+L(x_1)\tau  + \dots + L(x_{n-1})\tau^{n-1}$ for $x_i\in K$ with $N_{A_0}(1+x_1t+\dots +x_{n-1}t^{n-1})\not=0$.
 \item 
 $K^{(L(y_1)\tau + L(y_2)\tau^2 + \dots + L(y_{n-1})\tau^{n-1},g)}$  and  $N^0(f')$ is not empty, $N_{A_0/F}(y_1t+\dots+y_{n-1}t^{n-1})\not=0$, and $g\in {\rm Gl(V)}$.
 \\ Each set $N^0(f')$ yields a different type.
 \end{enumerate}

   \end{theorem}

 To compete a first attempt at a rough classification of $n$-dimensional algebras which have a  cyclic Galois field extension as unital heart, we now check when two algebras of the same type are isomorphic:

 \begin{theorem} \label{thm:cyclicGalois} 
 Let $K/F$ be a  cyclic Galois field extension of degree $n$, ${\rm Gal}(K/F)=\langle \tau\rangle$.
 \begin{enumerate}
 \item 
 For $y_1,z_1\in M$, $y_i,z_i\in K^\times$ for all $i$,  $N_{A_0}(y)\not=0$, $N_{A_0}(z)\not=0$, and $g,g'\in {\rm Gl}(V)$, we have
  $$K^{(1+L(y_1)\tau + L(y_2)\tau^2 + \dots + L(y_{n-1})\tau^{n-1},g)}\cong K^{(1+L(z_1)\tau + L(z_2)\tau^2 + \dots + L(z_{n-1})\tau^{n-1},g')}$$
 if and only if $y_1=z_1$, and exactly one of the following holds:
 \\ (i)  there exists $a\in F^\times$, such that  $g'=a g$ and  $z_i=y_i$,
   \\ (ii) there exists $a\in F^\times$, such that   $z_i=\tau^i(y_1^{-1}) y_1\tau(y_i)$, $g'=a \tau g\tau^{-1}L(y_1^{-1})$,
 \\ (iii) there exists $v\in K^\times$ and $\sigma\in {\rm Gal}(K/F)$, $\sigma\not\in \{ id,\tau\}$, such that $y_1=\tau(v) v^{-1}\sigma(y_1)$, $z_i=\tau^i(v) v^{-1}\sigma(y_i)$  and $g'=\sigma g\sigma^{-1}L(v)$.

 \item 
 For non-empty $N^0(f)=N^0(f')$,
  and $N_{A_0/F}(1+y_1t+\dots+y_{n-1}t^{n-1})\not=0$, $g,g'\in {\rm Gl(V)}$, we have
 $$K^{(1+L(y_1)\tau + L(y_2)\tau^2 + \dots + L(y_{n-1})\tau^{n-1},g)}\cong K^{(1+L(z_1)\tau + L(z_2)\tau^2 + \dots + L(z_{n-1})\tau^{n-1},g')}$$
 if and only if  there exists $v\in K^\times$ and $\sigma\in {\rm Gal}(K/F)$, such that $g'=\sigma g\sigma^{-1}L(v)$ and $z_i=\tau^i(v) v^{-1}\sigma(y_i)$ for all $i$.
 In particular, if there exists  $y_{i_k}\not=0$, such that $\tau^{i_k}$ generates ${\rm Gal}(K/F)$, then we have
   the smallest such $i_k$ with  $y_{i_k},z_{i_k}\in M$, so that
 $$K^{(1+L(y_1)\tau + L(y_2)\tau^2 + \dots + L(y_{n-1})\tau^{n-1},g)}\cong K^{(1+L(z_1)\tau + L(z_2)\tau^2 + \dots + L(z_{n-1})\tau^{n-1},g')}$$
  if and only if $y_{i_k}=z_{i_k}$,  and and exactly one of the following holds:
  \\ (i)  there exists $a\in F^\times$, such that  $g'=a g$ and  $z_i=y_i$,
   \\ (ii) there exists $a\in F^\times$ such that   $z_i=\tau^i(y_{i_k}^{-1}) y_1\tau^{i_k}(y_i)$ for all $i\not={i_k}$, $g'=a \tau^{i_k} g\tau^{-i_k}L(y_{i_k}^{-1})$,
    \\ (iii) there exists $v\in K^\times$ and  $\sigma\in {\rm Gal}(K/F)$, $\sigma\not\in \{ id,\tau^{i_k}\}$, such that $y_{i_k}=\tau(v) v^{-1}\sigma(y_{i_k})$, $z_i=\tau^i(v) v^{-1}\sigma(y_i)$  for all $i\not={i_k}$ and $g'=\sigma g\sigma^{-1}L(v)$.
  \item 
  For $x_i,w_i\in K$ with $N_{A_0}(x_1t+\dots+x_{n-1}t^{n-1})\not=0$, $N_{A_0}(w_1t+\dots+w_{n-1}t^{n-1})\not=0$,
   $$K^{(id,L(x_1)\tau +  \dots + L(x_{n-1})\tau^{n-1})}\cong K^{(id,L(w_1)\tau +  \dots + L(w_{n-1})\tau^{n-1})}$$
   if and only if  there exists $v\in K^\times$ and $\sigma\in {\rm Gal}(K/F)$, such that $ w_i= \sigma(x_i)\tau^i (v)$ for all $i$.
 \item 
 For $x_i,w_i\in K$ with $N_{A_0}(1+x_1t+\dots+x_{n-1}t^{n-1})\not=0$, $N_{A_0}(1+w_1t+\dots+w_{n-1}t^{n-1})\not=0$,
   $$K^{(id,1+L(x_1)\tau +  \dots + L(x_{n-1})\tau^{n-1})}\cong K^{(id,1+L(w_1)\tau +  \dots + L(w_{n-1})\tau^{n-1})}$$
 if and only if there exists  $\sigma\in {\rm Gal}(K/F)$, such that $ w_i= \sigma(x_i)$ for all $i$. 
 \item 
  For non-empty $N^0(f)=N(f')$,
  and $N_{A_0/F}(y_1t+\dots+y_{n-1}t^{n-1})\not=0$, $N_{A_0/F}(z_1t+\dots+z_{n-1}t^{n-1})\not=0$, $g,g'\in {\rm Gl(V)}$,
 $$K^{(L(y_1)\tau + L(y_2)\tau^2 + \dots + L(y_{n-1})\tau^{n-1},g)}\cong K^{(L(z_1)\tau + L(z_2)\tau^2 + \dots + L(z_{n-1})\tau^{n-1},g')}$$
 if and only if there are are $u,v\in K^\times$ and $\sigma\in {\rm Gal}(K/F)$, such that
 $g'=L(u^{-1})\sigma g\sigma^{-1}L(uv)$
  and $z_i=\tau^i(u)\tau^i(v) v^{-1}\sigma(y_i)$ for all  $i\in N^2(f)=N^2(f')$.

  \end{enumerate}
 \end{theorem}

 \begin{proof}
 By Lemma \ref{le:2.8}, two algebras mentioned in any of the cases of Theorem \ref{thm:cyclicGalois} are isomorphic if and only if there are $u,v\in K^\times$ and $\sigma\in {\rm Gal}(K/F)$, such that the critical relations in Lemma \ref{le:2.8} are satisfied.
 \\ \\(1) The critical relations are  $z_i=\tau^i(v) v^{-1}\sigma(y_i)$  and $g'=\sigma g\sigma^{-1}L(v)$, which implies
 $N_K(z_1)=N_K(y_1)$, hence $z_1=y_1$ by the definition of $M$ (Lemma \ref{le:M}).
 Using $z_1=y_1$ we go back to the equation
 $z_1=\tau(v) v^{-1}\sigma(y_1)$. If $\sigma=id$ this yields $v=a\in F^\times$. If $\sigma=\tau$ then $a=vy_1\in F^\times$. 
 Therefore the critical relations hold if and only if $z_1=y_1$ and
  there exists $a\in F^\times$ such that  $g'=a g$ and  $z_i=y_i$,
  or  there exists $a\in F^\times$ such that  $z_i=\tau^i(y_1^{-1}) y_1\tau(y_i)$ and $g'=a \tau g\tau^{-1}L(y_1^{-1})$, or  $y_1=\tau(v) v^{-1}\sigma(y_1)$,
  $z_i=\tau^i(v) v^{-1}\sigma(y_i)$  and $g'=\sigma g\sigma^{-1}L(v)$
   for some $\sigma\in {\rm Gal}(K/F)$, $\sigma\not\in \{ id,\tau\}$.
  \\ \\ (2) The critical relations are $g'=\sigma g\sigma^{-1}L(v)$
   and  $z_i=\tau^i(v) v^{-1}\sigma(y_i)$.
  If there exists  $y_{i_k}\not=0$, such that $\tau^{i_k}$ generates ${\rm Gal}(K/F)$, then
  $K\cong K^{(1+L(y_1)\tau + L(y_2)\tau^2 + \dots + L(y_{n-1})\tau^{n-1},g)}$
 for $y_{i_k}\in M$. Then $N_K(z_{i_k})=N_K(y_{i_k})$, hence $z_{i_k}=y_{i_k}$ by the definition of $M$.  This yields $v=a\in F^\times$ when $\sigma=id$ and $v=ay_{i_k}^{-1}$ when $\sigma=\tau^{i_k} $.
 \\\\ (3)
Here, the critical relations reduce to $u=1$ and
 $$L(w_1)\tau +\dots + L(w_{n-1})\tau^{n-1}=\sigma ( L(x_1)\tau +\dots + L(x_{n-1})\tau^{n-1}  )\sigma^{-1}L(v)$$
 $$ = \sigma L(x_1)\tau \sigma^{-1}L(v) + \sigma L(x_2)\tau^2  \sigma^{-1}L(v)+\dots + \sigma L(x_2)\tau^{n-1}  \sigma^{-1}L(v)$$
$$ = L( \sigma(x_1) \tau(v)) \tau + L(\sigma(x_2)\tau^2 (v) )\tau^2 +\dots + L(\sigma(x_2)\tau^{n-1} (v) )\tau^{n-1}.$$
By Corollary \ref{cor:2.9}, this is equivalent to
 $ w_i= \sigma(x_i)\tau^i(v)$ for all $i$.
\\\\ (4) 
Here, the critical relations reduce to $u=1$ and
$$1+L(w_1)\tau  +\dots + L(w_{n-1})\tau^{n-1} =  \sigma ( 1+L(x_1)\tau +\dots + L(x_{n-1})\tau^{n-1}  )\sigma^{-1}L(v)$$
$$ = L(v)+  \sigma L(x_1)\tau \sigma^{-1}L(v) +\dots + \sigma L(x_{n-1})\tau^{n-1}  \sigma^{-1}L(v)$$
$$ = L(v)+  L( \sigma(x_1) \tau(v)) \tau +\dots + L(\sigma(x_{n-1})\tau^{n-1} (v) )\tau^{n-1}.$$
 By Corollary \ref{cor:2.9}, this is equivalent to
$ L(v)=1$ and $ w_i= \sigma(x_i)\tau^i (v)$ for all $i$.
This is equivalent to $v=1$ and $ w_i= \sigma(x_i)\tau^i (v)= \sigma(x_i)$ for all $i$.
\\\\ (5)  Here, the critical relations cannot be simplified.

  \end{proof}

  \begin{corollary} \label{thm:cyclicGaloisprime} 
 Let $K/F$ be a  cyclic Galois field extension of prime degree $n$ with ${\rm Gal}(K/F)=\langle \tau\rangle$,  $\sigma\in {\rm Gal}(K/F)$.
 \begin{enumerate}
 \item 
 For $y_1,z_1\in M$, $y_i,z_i\in K^\times$ with $N_{A_0}(y)\not=0$, $N_{A_0}(z)\not=0$, and $g,g'\in {\rm Gl}(V)$, we have
  $$K^{(1+L(y_1)\tau + L(y_2)\tau^2 + \dots + L(y_{n-1})\tau^{n-1},g)}\cong K^{(1+L(z_1)\tau + L(z_2)\tau^2 + \dots + L(z_{n-1})\tau^{n-1},g')}$$
 if and only if $y_1=z_1$, and exactly one of the following holds:
 \\ (i)  there exists $a\in F^\times$, such that  $g'=a g$ and  $z_i=y_i$ for all $i$,
   \\ (ii) there exists $a\in F^\times$ such that   $z_i=\tau^i(y_1^{-1}) y_1\tau(y_i)$ for all $i$, $g'=a \tau g\tau^{-1}L(y_1^{-1})$,
    \\ (iii) there exists $v\in K^\times$ and $\sigma\in {\rm Gal}(K/F)$, $\sigma\not\in \{ id,\tau\}$, such that $y_1=\tau(v) v^{-1}\sigma(y_1)$, $z_i=\tau^i(v) v^{-1}\sigma(y_i)$ for all $i$,  and $g'=\sigma g\sigma^{-1}L(v)$.
 \item 
  For non-empty $N^0(f)=N^0(f')$, and $N_{A_0/F}(1+y_1t+\dots+y_{n-1}t^{n-1})\not=0$, $g,g'\in {\rm Gl(V)}$, and  $i_k$ be such that $y_{i_k},z_{i_k}\in M$
  (i.e. this is the smallest $i_k$ be such that $y_{i_k}\not=0$), we have
 $$K^{(1+L(y_1)\tau + L(y_2)\tau^2 + \dots + L(y_{n-1})\tau^{n-1},g)}\cong K^{1+L(z_1)\tau + L(z_2)\tau^2 + \dots + L(z_{n-1})\tau^{n-1},g')}$$
  if and only if $y_{i_k}=z_{i_k}$,  and exactly one of the following holds:
 \\ (i)  there exists $a\in F^\times$, such that  $g'=a g$ and  $z_i=y_i$ for all $i$,
   \\ (ii) there exists $a\in F^\times$ such that   $z_i=\tau^i(y_{i_k}^{-1}) y_1\tau^{i_k}(y_i)$ for all $i\not={i_k}$, $g'=a \tau^{i_k} g\tau^{-i_k}L(y_{i_k}^{-1})$,
    \\ (iii) there exists $v\in K^\times$ and  $\sigma\in {\rm Gal}(K/F)$, $\sigma\not\in \{ id,\tau^{i_k}\}$, such that $y_{i_k}=\tau(v) v^{-1}\sigma(y_{i_k})$, $z_i=\tau^i(v) v^{-1}\sigma(y_i)$  for all $i\not={i_k}$ and $g'=\sigma g\sigma^{-1}L(v)$.
  \item 
   For $x_i,w_i\in K$ with $N_{A_0}(x_1t+\dots+x_{n-1}t^{n-1})\not=0$, $N_{A_0}(w_1t+\dots+w_{n-1}t^{n-1})\not=0$,
   $$K^{(id,L(x_1)\tau +  \dots + L(x_{n-1})\tau^{n-1})}\cong K^{(id,L(w_1)\tau +  \dots + L(w_{n-1})\tau^{n-1})}$$
   if and only if  there exists $v\in K^\times$ and  $\sigma\in {\rm Gal}(K/F)$, such that $ w_i= \sigma(x_i)\tau^i (v)$ for all $i$.
 \item 
 For $x_i,w_i\in K$ with $N_{A_0}(1+x_1t+\dots+x_{n-1}t^{n-1})\not=0$, $N_{A_0}(1+w_1t+\dots+w_{n-1}t^{n-1})\not=0$,
   $$K^{(id,1+L(x_1)\tau +  \dots + L(x_{n-1})\tau^{n-1})}\cong K^{(id,1+L(w_1)\tau +  \dots + L(w_{n-1})\tau^{n-1})}$$
 if and only if there exists  $\sigma\in {\rm Gal}(K/F)$ such that $ w_i= \sigma(x_i)$ for all $i$.
 \item 
  For non-empty $N^0(f)=N^0(f')$, $N_{A_0/F}(y_1t+\dots+y_{n-1}t^{n-1})\not=0$, $N_{A_0/F}(z_1t+\dots+z_{n-1}t^{n-1})\not=0$, $g,g'\in {\rm Gl(V)}$,
 $$K^{(L(y_1)\tau + L(y_2)\tau^2 + \dots + L(y_{n-1})\tau^{n-1},g)}\cong K^{(L(z_1)\tau + L(z_2)\tau^2 + \dots + L(z_{n-1})\tau^{n-1},g')}$$
 if and only if there are are $u,v\in K^\times$ and  $\sigma\in {\rm Gal}(K/F)$, such that
 $g'=L(u^{-1})\sigma g\sigma^{-1}L(uv)$
  and $z_i=\tau^i(u)\tau^i(v) v^{-1}\sigma(y_i)$ for all  $i\in N^2(f)=N^2(f')$.

 \end{enumerate}
  \end{corollary}

\bigskip
\emph{Acknowledgments:}
This paper was written while the author was a visitor at the University of Ottawa. She acknowledges support for her stay from  the Centre de Recherches Math\'ematiques for giving a colloquium talk,
  and from Monica Nevins' NSERC Discovery Grant RGPIN-2020-05020. She would like to thank the Department of Mathematics and Statistics for its hospitality and the congenial atmosphere, and in particular M. Nevins for lots of fruitful discussions.



\begin{thebibliography}{1}

\bibitem{B} U. Bekbaev, \emph{Complete classification of two-dimensional algebras over any basic field}. AIP Conf. Proc. 2880, 030001 (2023)
\\
\verb#https://doi.org/10.1063/5.0165726#



\bibitem{D} E. Darp\"o, \emph{Isotopes of Hurwitz algebras}. Mediterr. J. Math. (2018) 15:52
\verb#https://doi.org/10.1007/s00009-018-1095-y#



\bibitem{D1} L. E. Dickson, \emph{On finite algebras.} Nachrichten Ges. Goettingen (1950), 358-393.
\bibitem{D2} L. E. Dickson, \emph{Linear algebras in which division is always uniquely possible}. Trans. AMS 7 (1906), 370-390.
\bibitem{D3} L. E. Dickson, \emph{On linear algebras.} AMS Monthly 13 (1906), 201-205.
\bibitem{D4} L. E. Dickson, \emph{On triple algebras and ternary forms.} Bull. AMS 14 (1908), 160-168.

\bibitem{G} P.~Gabriel \emph{Finite representation type is open}. Representations of Algebras. Springer Lecture Notes 488, 132–155 (1975).

\bibitem{HP} M. H\"ubner, H. P.~Petersson, \emph{Two-dimensional real division algebras revisited.}
Beitr\"age Algebra Geom. 45(1), 29–36 (2004)

\bibitem{J96} N.~Jacobson,
``Finite-dimensional division algebras over fields.'' Springer Verlag,
Berlin-Heidelberg-New York, 1996.

\bibitem{KI} I.~Kaplansky, \emph{Three-dimensional division algebras.} I. J. Algebra 40 (1976), 384-391.
\bibitem{KII} I.~Kaplansky, \emph{Three-dimensional division algebras. II.} Houston Journal of Mathematics 1 (1975), 63-79.

\bibitem{Ka} I. Kaygorodov, \emph{Non-associative algebraic structures: classification and structure.} Commun. Math. 32 (3) (2024), 1–62.

\bibitem{KV} I. Kaygorodov, Y. Volkov, \emph{The variety of two-dimensional algebras
over an algebraically closed field.} Canad. J. Math. Vol. 71 (4) (2019) 819-842.

\bibitem{KMRT}  M.-A.~Knus, A.~Merkurjev,, M.~Rost, J.-P.~Tignol,
``The Book of Involutions'', AMS Coll. Publications, vol. 44 (1998).

\bibitem{McC} K.~McCrimmon, \emph{Homotopes of alternative algebras}. Math. Ann. 191 (1971), 253-262.

\bibitem{Ma} G.~Mazzola, \emph{The algebraic and generic classification of associative algebras of dimension 5}. Manuscripta Math. 27 (1979), 81-101.

\bibitem{M1} G.~Menichetti, \emph{Algebre tridimensionali su un campo di Galois.}
Ann. Mat. Pura Appl. 97 (4) (1973), 283-302.

\bibitem{M2} G.~Menichetti, \emph{On a Kaplansky conjecture concerning three-dimensional division algebras over a finite field}.
J.  Algebra 47 (2) (1977),  400-410.


\bibitem{M} G.~Menichetti,  \emph{Sopra una classe di quasicorpi distributivi di ordine finito.}
Atti Accad. Naz. Lincei Rend. Cl. Sci. Fis. Mat. Natur. (8), 59 (5) (1975), 339-348. 


\bibitem{P} H. P.~Petersson, {\it The classification of two-dimensional nonassociative algebras},
Result. Math. 37 (2000), 120-154.

\bibitem{PS} H. P.~Petersson, M. Scherer, \emph{The number of non-isomorphic two-dimensional algebras over a finite field.}
Result. Math. 45 (2004), 1-2, 137–152.



\bibitem{RRB} I. S. Rakhimov, I. M. Rikhsiboev and W. Basri, \emph{Complete lists of low dimensional complex associative algebras.} 2009.
\verb#https://arxiv.org/pdf/0910.0932.pdf#

\bibitem{Sch} R. D.~Schafer, ``An Introduction to Nonassociative Algebras.'' Dover Publ. Inc., New York, 1995.

\bibitem{S13} A. Steele, {\it Some new classes of division algebras and potential applications to space-time block coding}. PhD Thesis, University of Nottingham, 2013.
\verb#http://eprints.nottingham.ac.uk/13934/1/PhdthesisFinal.pdf#

\end{thebibliography}
\end{document}